\documentclass[11pt]{article}
\usepackage[english]{babel}
\usepackage{amsthm}
\usepackage{epsfig}
\usepackage{xcolor}

\newtheorem{theorem}{Theorem}[section]
\newtheorem{lemma}[theorem]{Lemma}
\begin{document}
%


\title{\bf Dynamical Clustering and its Limiting Differential Equations}
\author{Paul A. Milewski, Esteban G. Tabak}

\maketitle \pagestyle{plain} \pagenumbering{arabic}

\begin{abstract}
A novel, non-parametric clustering algorithm is developed, based on flows in the space of soft assignments $P_k^i$, representing the probability that point $x^i$ belongs to class $k$. These flows have two drivers: a nonlinear reaction component  that relaxes $P$ to an evolving Bayesian-like posterior, and a diffusive component that relaxes $P$ to its local average, with a diffusivity $\nu$ that adapts dynamically so as to balance the two components. The methodology extends to semi-supervised classification, where a subset of the labels are known. In the limit of infinitely many observations, it yields a set of non-standard reaction-diffusion equations, which produces sharp boundaries between species that separate spontaneously into well-balanced domains.  Including more than one diffusive network opens the way to a broader class of applications, including the detection of regime changes in time series and a clustering procedure based on numerous features that defeats the curse of dimensionality. In the continuous limit, this extension results in a novel, non-local class of diffusive operators.

\end{abstract}

\section{Introduction}

Assigning labels $y$ to observations $x$ is a central task in data science \cite{Bishop, Hastie}, a typical instance being the assignment of a diagnostic to the results of a clinical test. Supervised learning, which bases this assignment on a training set of pairs $\{x^i, y^i\}$, is named classification or regression depending on whether the label is categorical or continuous. In unsupervised learning, the training set consists only of the observations $\{x^i\}$, and the assignment of labels acquires different connotations: factor discovery when the $y$ are latent variables that explain part of the variability in $x$, and dimension or complexity reduction when $y$ provides an economical or interpretable representation for $x$,  named clustering when the labels $y$ are categorical. Semisupervised learning combines the two previous scenarios, with data that includes a set of observations $\{x^i\}\ i \in \{1, 2, \ldots, m\}$ and only a partial list of labels $y^i, \ i \in I = \{i_1, i_2, \ldots, i_n\}$, $n < m$. Then the assignment of labels $y$ to the complement of the labeled set $I$ is based both on the available pairs and on the distribution underlying the observations.


This article proposes a new methodology for the unsupervised and semisupervised learning of categorical labels $y \in \{y_1, \ldots, y_K\}$, i.e. clustering and semisupervised classification, based on flows in probability space. The labeling proposed adopts the form of a soft assignment $P_k^i = P(y_k|x^i)$. Here the $\{x^i\}$ are samples of the distribution $\rho$ of a multidimensional variable $X$ in a metric space, the $\{y_k\}$ are the labels, and $P_k^i$ represents the probability that sample $x^i$ has label $y_k$, satisfying
$$ 0 \le P_k^i \le 1, \quad \sum_{k=1}^K P_k =1. \label{constraint}$$
In the semisupervised case, the $P_k^i$ are provided for $i \in I$. 

In the method proposed, the $P_k^i$ (either all of them for unsupervised problems, or the unknown set for semisupervised problems) are initialized to non-informative values and evolve in an algorithmic time $t$ through a set of differential equations that converge to a final assignment.



The dynamics for the $\{P_k^i\}(t)$ has two components. The first component is a continuous non-parametric Bayesian-like update, which relaxes the $\{P_k^i\}$ toward a posterior $\{\tilde{P}_k^i\}$. The determination of the posterior uses the current $\{P_k^i\}$ both as a Bayesian prior and as weights to build an estimation for the distribution $\rho_k(x)$ for each class $k$. The non-parametric nature of this estimation makes this update insensitive to any metric information in the space $X$ of observations. In order to restore this information, a second, diffusive component is added to the dynamics, which relaxes $P_k^i$ toward the values of its neighbors in $X$-space. This approach turns out to be very rich, permitting for instance the use of more than one notion of neighborhood, which when clustering time series, allows the $\{P_k^i\}$ to diffuse both in $X$ and in the (real) time variable $T$, and defeats the curse of dimensionality in classification problems with numerous features.

The procedure requires as input:
\begin{enumerate}
  \item Observations: samples $\{x^i \in X\}, \ i \in \{1, 2, \ldots, m\}$ of the underlying distribution $\rho(x)$,
  \item The cardinality $K$ of the categorical labels $Y = \{y_k\}$ sought,
  \item For semisupervised problems, the probability $P_k^i \in [0, 1]$ for $i \in I \subset \{1, 2, \ldots, m\}$ that sample $x^i$ has label $y_k$,
  \item One or more notions of distance $d_{ij}^l$ between pairs of points $(x^i, x^j)$ in $X$. If more that one distance $d^l$ is provided, also a weight $\lambda^l$ assigned to each.
\end{enumerate}
It provides as output $\{P_k^i\} \in [0, 1]$, with $\sum_{k=1}^K P_k^i = 1$. Even though the $\{P_k^i\}$ can be conceptualized as probabilities, the default choice for the algorithm's parameters for clustering returns hard assignments $P_k^i \in \{0, 1\}$, with the intermediate soft assignments $P_k^i(t)$ used only as computational tools, much as in interior point methodologies for optimization. 

The first part of the paper derives this system of data-driven differential equations. In the second part we write the 
resulting system for the evolution of the $P_k^i(t)$ in a continuum limit as the number of observations grows. This yields a non-local system of reaction-diffusion equations 
$$
\frac{\partial P_k}{\partial t} =  R_k + \nu D_k.
$$
Here $P_k(x, t)$ is the continuum limit of $P_k^i(t) \rightarrow P_k(x^i, t)$, satisfying at all times the conditions
$$ 0 \le P_k(x, t) \le 1, \quad \sum_{k=1}^K P_k(x, t) = 1. $$
The probability density $\rho(x)$ represents the distribution underlying the samples $\{x^j\}$. The right-hand side of the system includes a non-local reaction term $R$ derived from the Bayesian update, and a diffusion term enforcing a distance metric. The reaction term is 
$$ R_k(x, t) = \left(\frac{\frac{P_k}{Z_k}}{\sum_{h=1}^K \frac{{P_h}^2}{Z_h}}  - 1  \right) P_k $$
where
$$ Z_k(t) = \int P_k(x, t) \ \rho(x) \ dx $$
represents the total probability of class $k$. The diffusive term is
$$ D_k(x, t) = \frac{1}{\rho} \nabla \cdot \left(\rho \nabla P_k\right),$$
with a diffusivity $\nu$ that is adjusted over time through
$$ \nu = \sqrt{\frac{\sum_k \int R_k^2 \rho\ dx}{\sum_k \int D_k^2 \rho\ dx}}, $$
so as to strike an adaptive balance between the reactive and diffusive component of the system.

This system displays interesting dynamics that makes it an object of study in its own right, in addition to shedding light on the nature of the clustering algorithm.

Even though the methodology proposed relates in different ways to $k$-means \cite{KM}, to expectation-maximization density estimation through mixtures \cite{EM},  to diffusive maps \cite{DM} and to normalizing-flow-based clustering procedures \cite{NF}, it is fundamentally different from them and, to the extent of our knowledge, from existing procedures for clustering and semi-supervised classification.
The system in the continuous limit can be compared to more standard reaction-diffusion systems, such as those proposed as models for species competition \cite{GSC, SC}. In contrast to those, the boundaries between classes that it yields are sharp and classes separate spontaneously into well-balanced domains even in the absence of boundary conditions.

\section{Evolutionary equations for clustering assignments} 

In order to motivate and derive the system ruling the time evolution of the $P_k^i(t)$, we first propose a parametric, continuous in time k-means-like clustering algorithm based on a gradual relaxation to Bayesian updates. Extending this procedure to a nonparametric setting loses information on the geometry of the observations $x$, which is then  restored through the addition of a diffusive component to the flows.

\subsection{Parametric time-evolution clustering} \label{ParamEq}

This subsection describes a Bayesian-driven, parametric algorithm for clustering that motivates the non-parametric  methodology that follows.
In this algorithm, the $P_k^j(t)$ are used as weights to estimate a probability density $\rho_k(x, t)$ per class.
For instance propose that the $\rho_k$ are Gaussians, use the weights
$$ w_k^j (t) = \frac{P_k^j}{\sum_{j=1}^m P_k^j} $$
corresponding to the fractional contribution of $x_j$ to class $k$ to estimate the mean and covariance matrix of each class:
$$
 \mu_k(t) = \sum_j w_k^j x^j, \quad \Sigma_k(t) = \sum_j w_k^j \left(x^j-\mu_k\right) \left(x^j-\mu_k\right)^{\hbox{\small{T}}},
$$
and set
\begin{equation}
 \rho_k(x,t) \sim \mathcal{N}(\mu_k, \Sigma_k).
 \label{Gaussians}
\end{equation}
Since the $w_k^j$ depend on $t$, all quantities which follow do too, which we will not indicate explicitly in what follows. Then, considering the current $P_k^j$ as a prior, one can compute the posterior through Bayes' formula:
$$ \tilde{P}_k^j = \frac{P_k^j \ \rho_k\left(x^j\right)}{\sum_{h=1}^K P_h^j \ \rho_h\left(x^j\right)} $$
and relax $P_k^j$ smoothly towards $\tilde{P}_k^j$ with $\frac{d P_k^j}{dt} = \tilde{P}_k^j - P_k^j$. This gives rise to the nonlinear system of differential equations
\begin{equation}
 \frac{d P_k^j}{dt} = \left(\frac{\rho_k\left(x^j\right)}{\sum_{h=1}^K P_h^j \ \rho_h\left(x^j\right)}  - 1  \right) P_k^j.
\label{Pt1}
\end{equation}
The initial conditions can be chosen as the uniform distribution $P_k^j = \frac{1}{K}$  (or another prior if available) plus a small random perturbation required to break the symmetry among classes. Equation (\ref{Pt1}) is a system of $m \times K$ ordinary differential equations where the $K$ constraints (\ref{constraint}) can be used to reduce the size of the system to $(m-1)\times K$. 
An approximate solution may be obtained by implementing a simple forward Euler approximation
\begin{equation}
 \frac{\left(P_k^j\right)^{n+1}-\left(P_k^j\right)^n}{\Delta t} =
 \left(\frac{\rho_k\left(x^j, t_n\right)}{\sum_{h=1}^K \left(P_h^j\right)^n \ \rho_h\left(x^j, t_n\right)}  - 1  \right) \left(P_k^j\right)^n,
 \label{Pt1_disc}
\end{equation}
where $\left(P_k^j\right)^n$ is the approximate solution at algorithmic time $t=t_n$. Figure \ref{fig:Gaussians} displays a successful application of this algorithm to a situation where the distributions underlying three clusters are indeed Gaussian. 

\begin{figure}[!htpb]
\begin{center}
\begin{tabular}{lll}
\includegraphics[width=.49\textwidth]{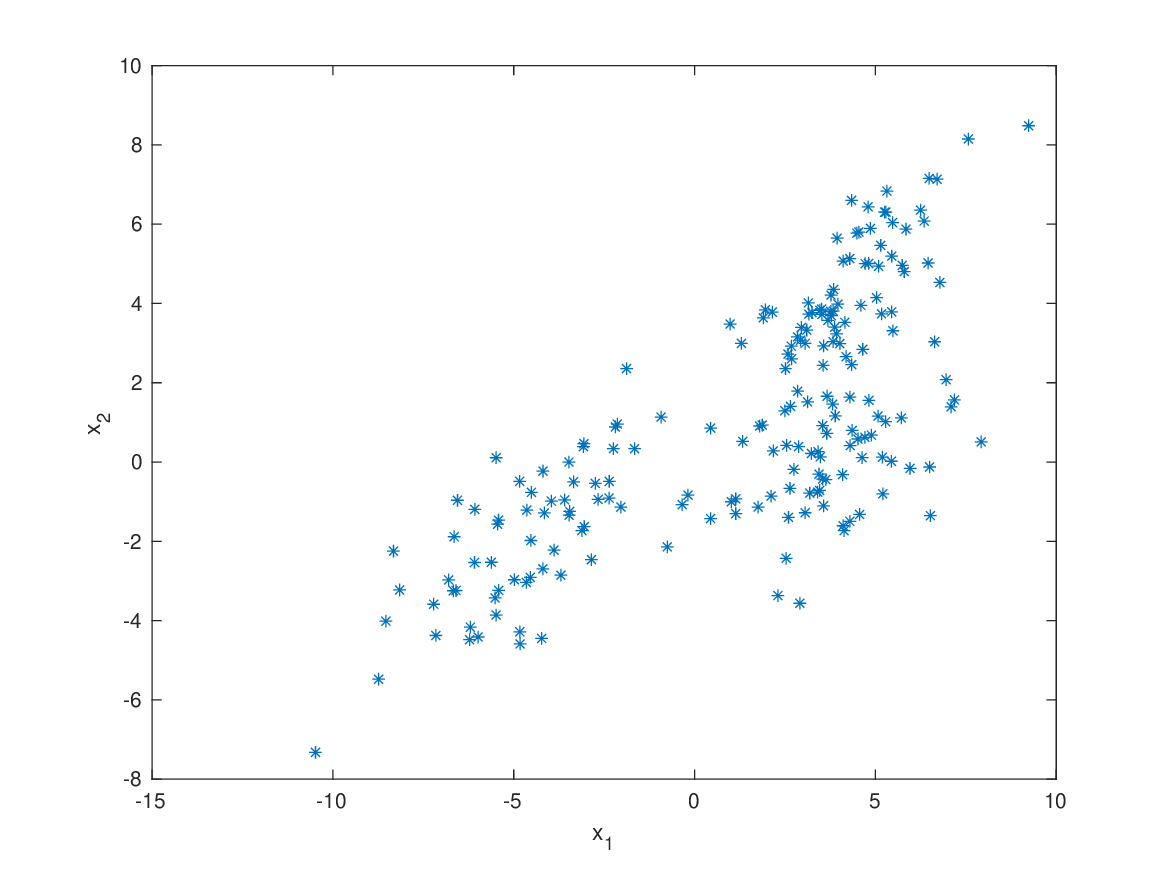} &
\includegraphics[width=.49\textwidth]{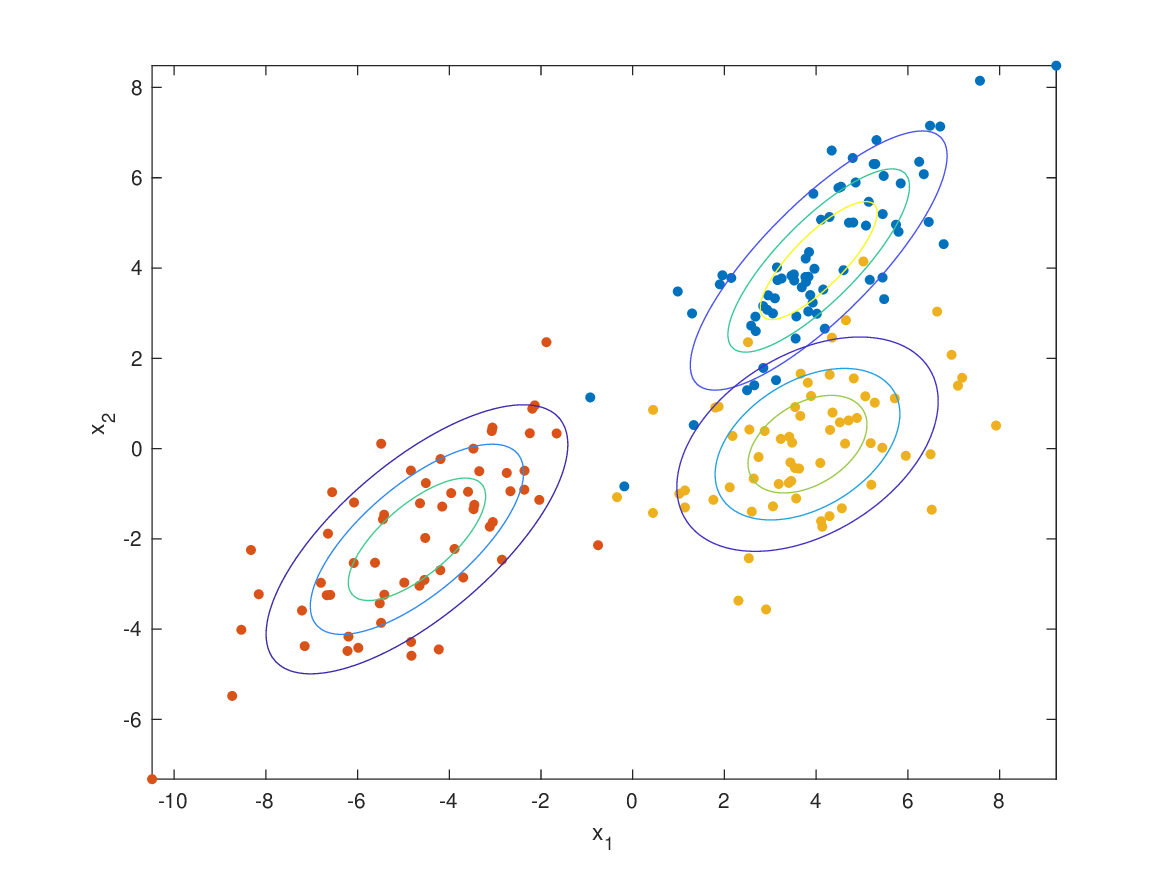} \\
\includegraphics[width=.49\textwidth]{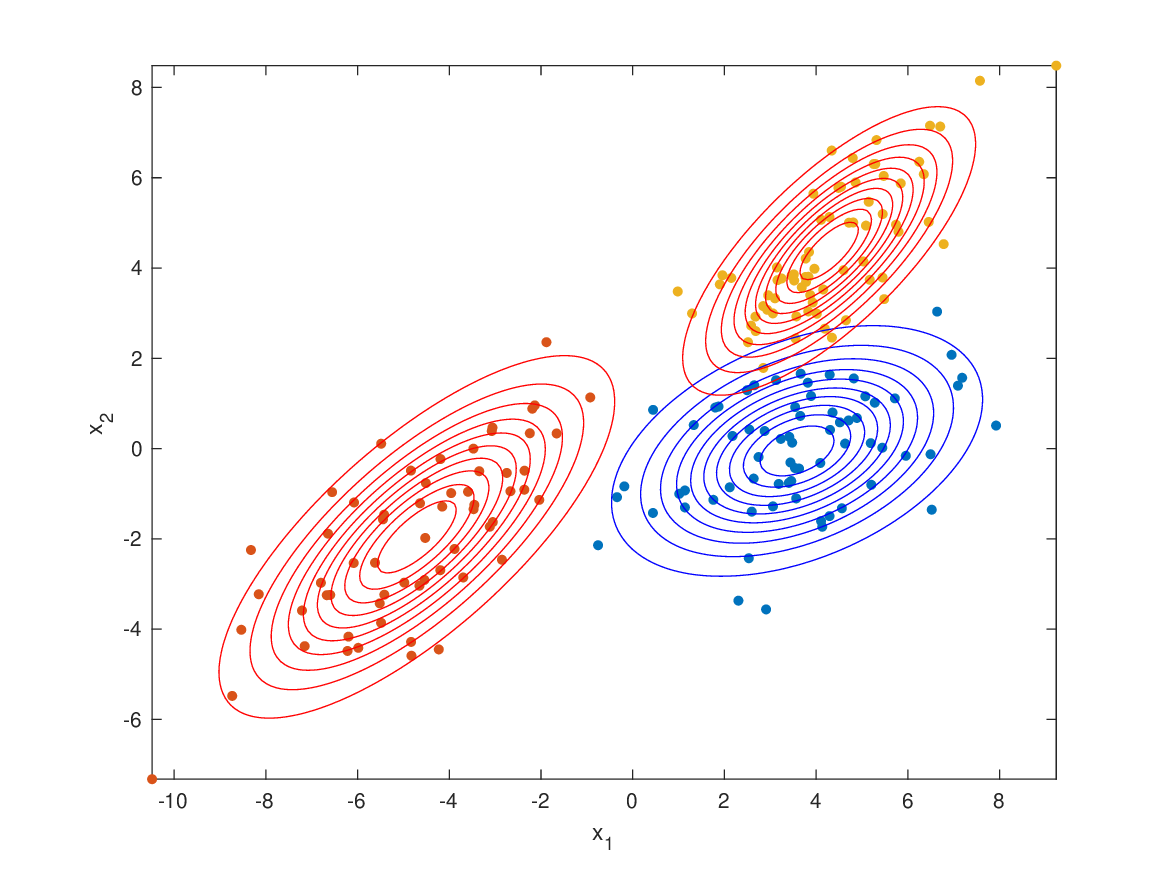} &
\includegraphics[width=.49\textwidth]{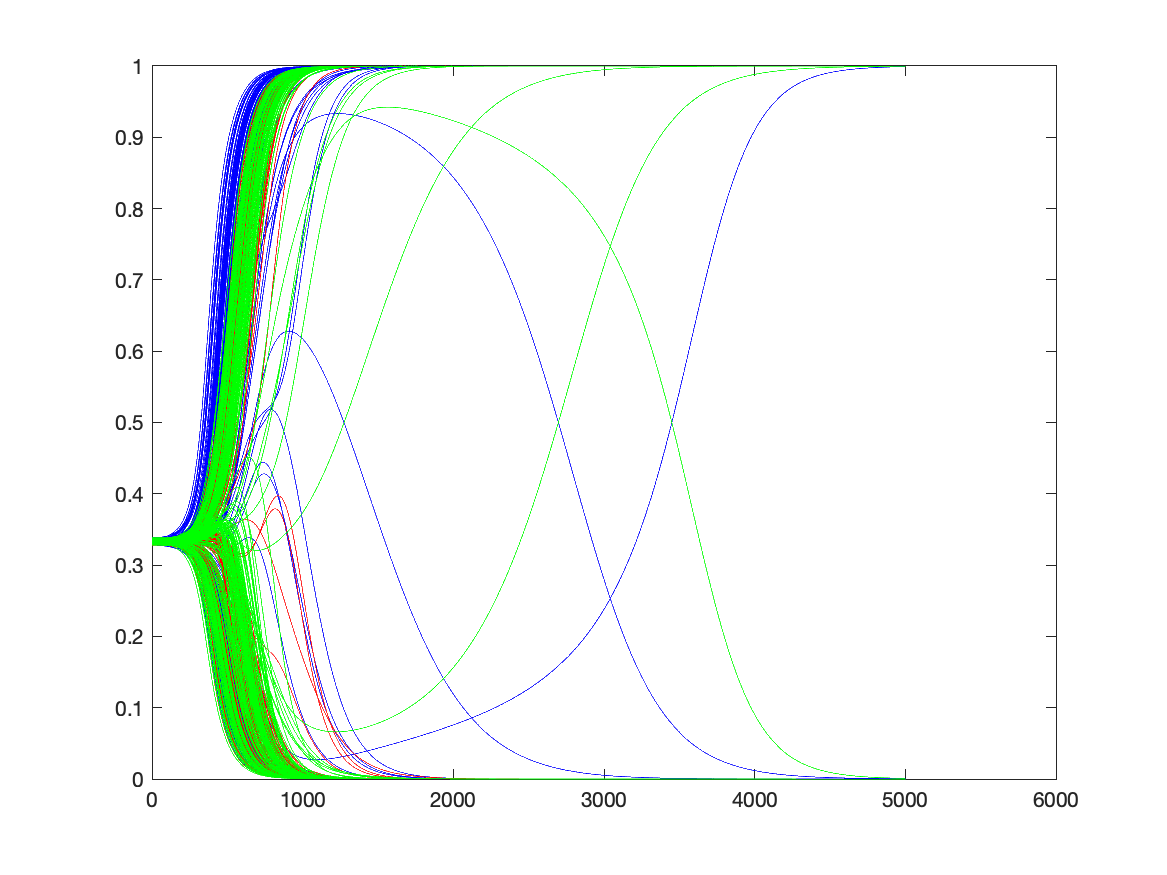} \\
\end{tabular}
\caption{Three Gaussian clusters. Top left: sample points from three Gaussian distributions in the plane. Top right: same samples, colored by class and overlayed with contours of the generating distributions. Bottom left: class assignment by the algorithm and corresponding estimated distributions. Bottom right: probabilities $P_k^j(t)$. Notice that these all asymptote to zero or one, i.e. all final assignments are rigid. Notice also how some samples are --correctly-- assigned to classes whose center is not the closest in Euclidean norm, something that $k$-means cannot achieve.
}
\label{fig:Gaussians}
\end{center}
\end{figure}

Before switching to a non-parametric setting, it may be useful to place this parametric algorithm in context by comparing it to two standard related algorithms, k-means and expectation-maximization (EM). 
In $k$-means,
\begin{enumerate}
\item The assignment to classes is hard at each step, so $P_k^j \in \{0,1\}$ and $w_k^j \in \left\{0,\frac{1}{m_k}\right\}$, where $m_k$ is the number of samples currently assigned to class $k$.

\item Only the means $\{\mu_k\}$ are estimated, not the covariance matrices $\{\Sigma_k\}$. 
\item Instead of estimating $\tilde{P}_k^j$ through Bayes theorem, one assigns $\tilde{P}_k^j = 1$ to the class $k$ with mean $\mu_k$ closest to $x^j$.
Thus points may be assigned to clusters under whose distribution they have small probability.

\item One adopts $\left(P_k^j\right)^{n+1} = \left(\tilde{P}_k^j\right)^n$, rather than relaxing $P_k^j$ toward $\tilde{P}_k^j$ through a time-dependent flow. 

\end{enumerate}

An alternative comparison contrasts the evolution (\ref{Pt1}) above with EM based density estimation by Gaussian mixtures. This does not cluster the data, but produces an estimation of $\rho(x)$ consisting of a mixture of $K$ Gaussians with weights $p_k$, i.e. find $p_k,\tilde{\mu}_k, \tilde{\Sigma}_k$ such that 
$$ \rho(x) = \sum_k p_k \; \mathcal{N}(\tilde{\mu}_k, \tilde{\Sigma}_k).$$
The key difference between the two procedures is the prior used in Bayes formula. Whereas the algorithm described uses the current value of $P_k^j$ as prior, EM uses $p_k$ (i.e. the global probability of sampling from class $k$). The two are related by
$$ p_k = \frac{1}{m} \sum_j P_k^j . $$

Thus EM produces an estimation of the global density $\rho(x)$, consisting of a mixture of $K$ Gaussians from which soft assignments can be inferred, while the algorithm above produces a characterization, where each sample $i$ has its own hard assignment to a class (i.e $P_k^j$ converging to either zero or one), in addition to an estimate of each cluster density $\rho_k(x)$. The hard assignment results from the nonlinear feedback in (\ref{Pt1}).

The choice of Gaussian distributions in our first example limits the procedure's applicability, as the clusters sought may have arbitrary distributions, with nontrivial shape and topology. Figure \ref{fig:Spirals} (top-left) shows an example where the two true clusters spiral around each other, which the algorithm cannot possibly capture (top-right). We may remedy this through a more general parametric algorithm using kernel density estimators, which replace the Gaussian distributions $\rho_k$ in (\ref{Gaussians}) by 
$$ \rho_k(x,t) = \sum_l w_k^l (t) \; K\left(x, x^l\right), \quad K(x, y) \ge 0, \quad \int_x K(x, y)\ dx = 1,$$
with a kernel $K$ of choice. For example, we choose isotropic Gaussians with tunable bandwidth $\sigma$ 
$$K=\frac{1}{\sqrt{2\pi\sigma}}e^{|x-x_l|^2/\sigma^2}$$ 
in the example of Figure 2. This kernel-based extension is not very costly, as the kernels are only applied to pairs of sample points, so the matrix
$$ K_j^l = K\left(x^j, x^l\right) $$
can be pre-computed before implementing the flow in (\ref{Gaussians}),  which requires repeated evaluations of
$$ \rho_k(x^j) = \sum_l w_k^l \; K\left(x^j, x^l\right) = \sum_l w_k^l \; K_j^l.$$

\begin{figure}[!htpb]
\begin{center}
\begin{tabular}{lll}
\includegraphics[width=.45\textwidth]{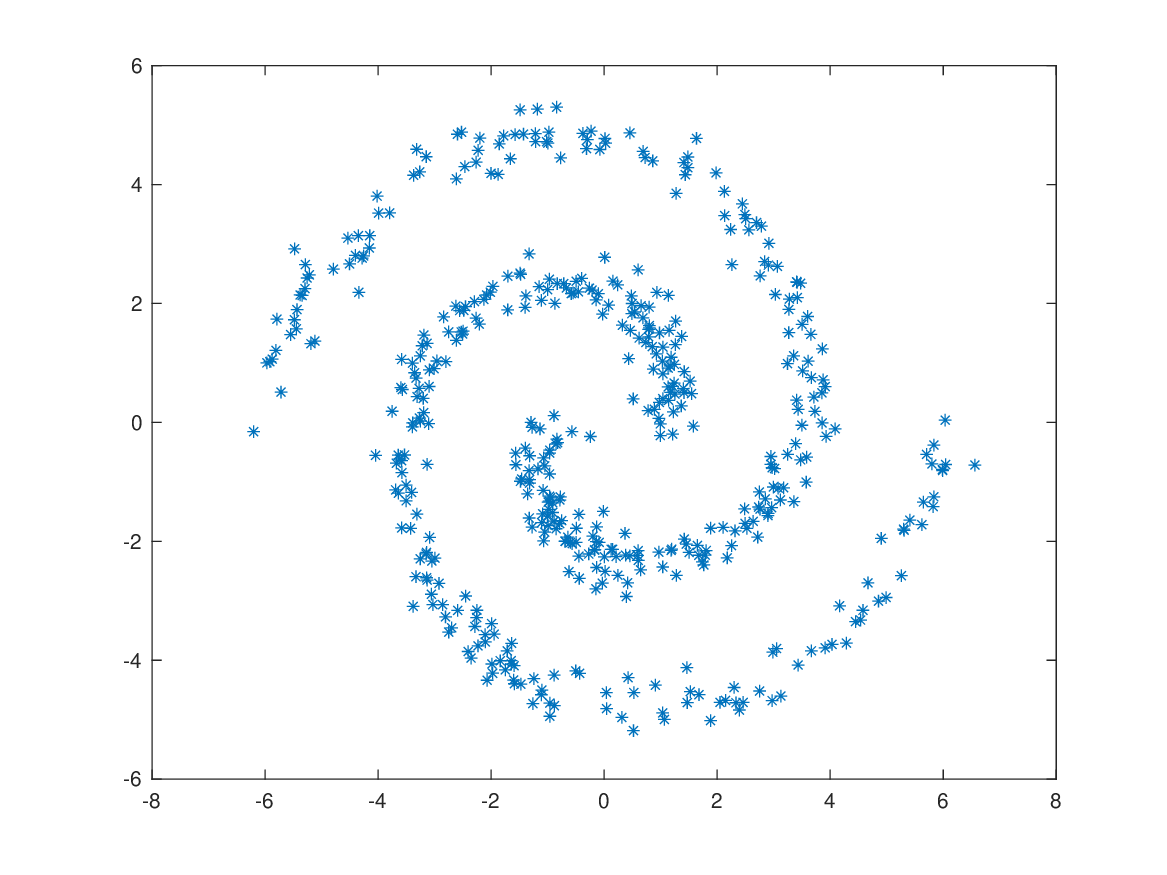} &
\includegraphics[width=.45\textwidth]{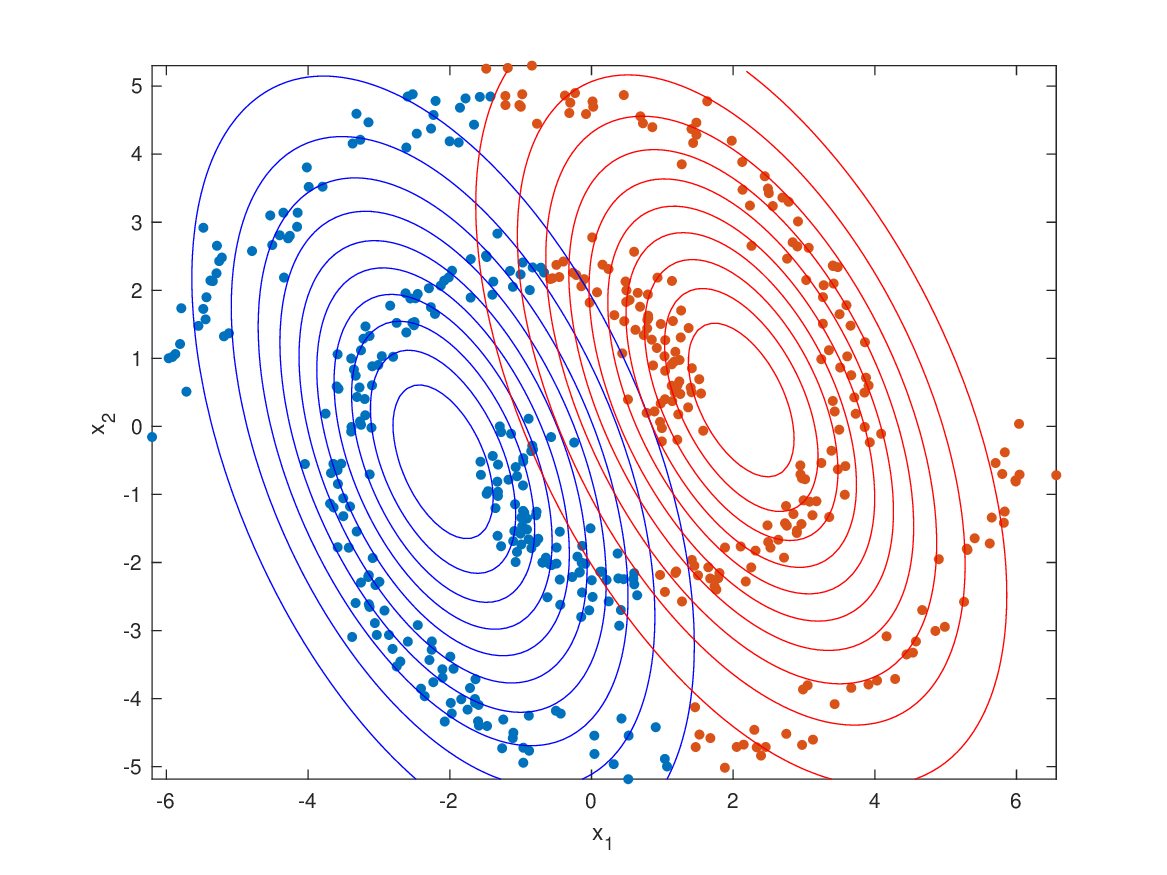} \\
\includegraphics[width=.45\textwidth]{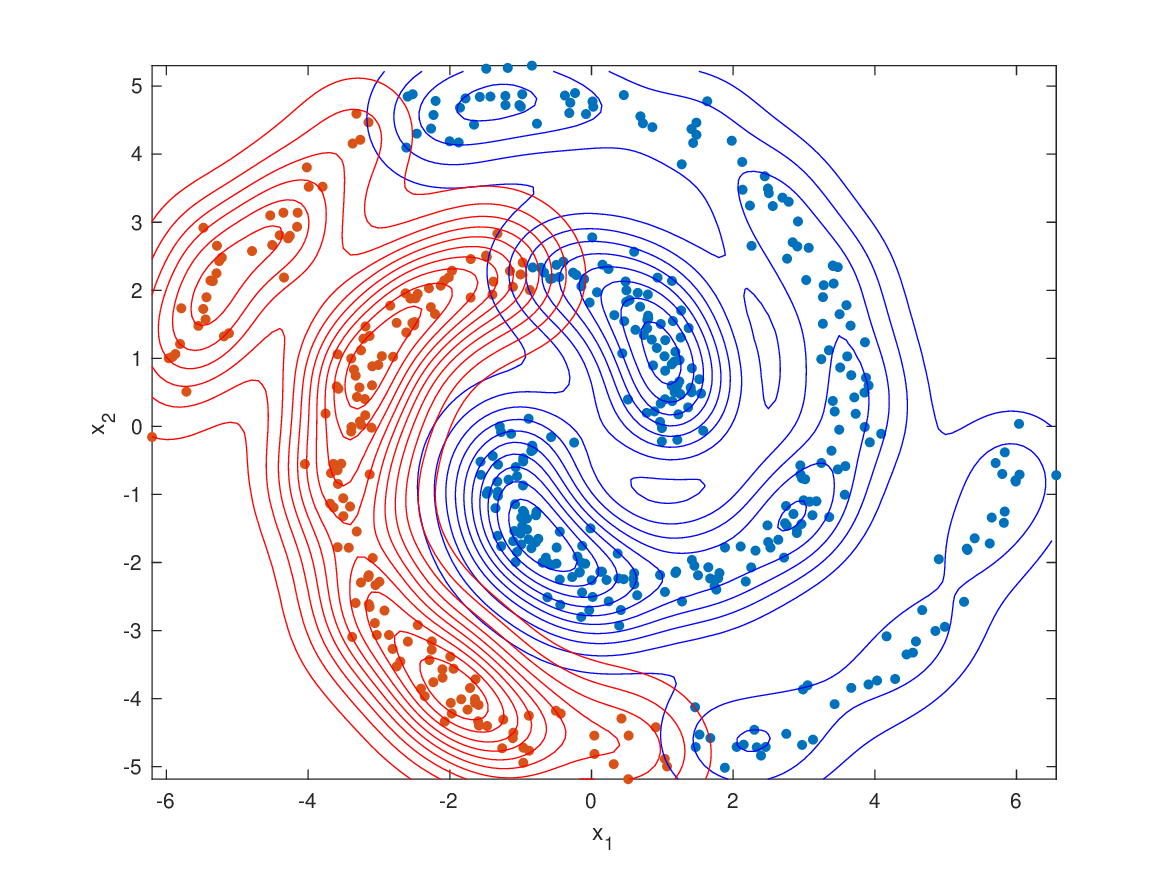} &
\includegraphics[width=.45\textwidth]{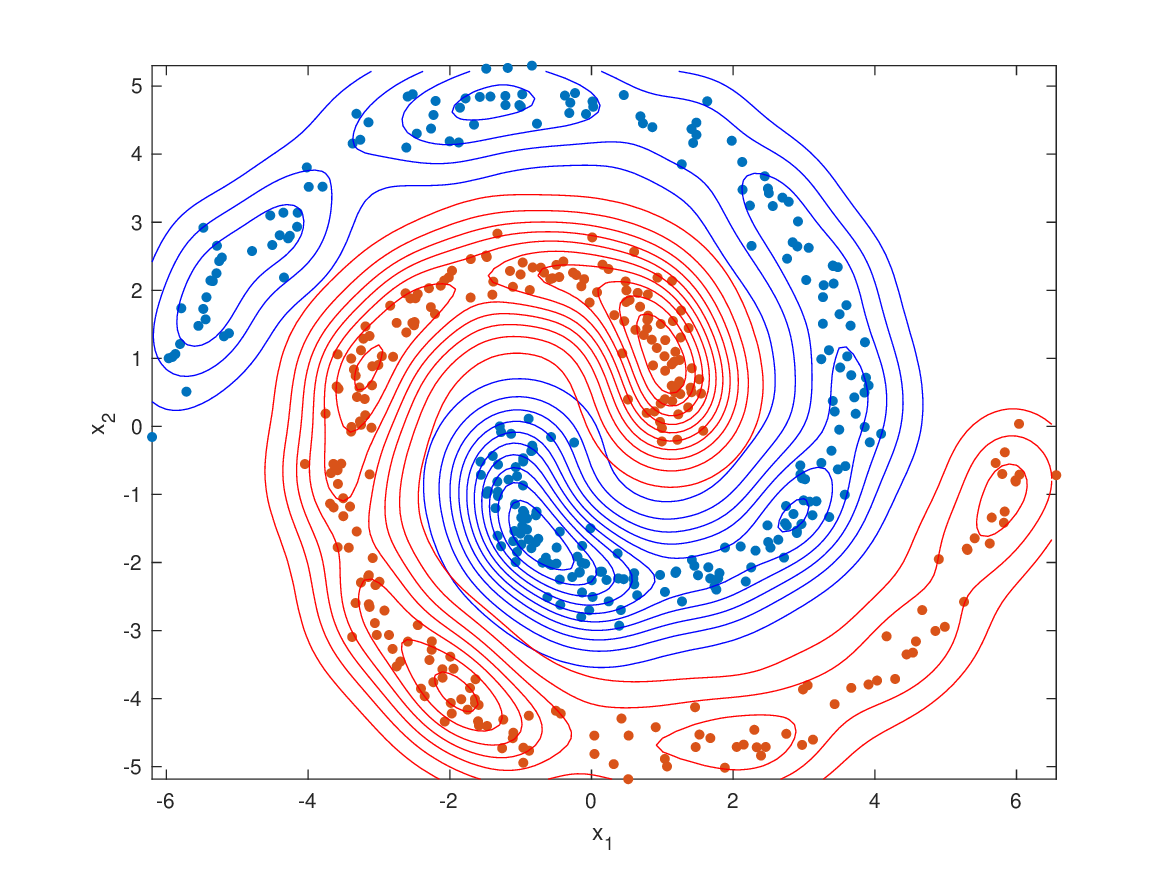} \\
\includegraphics[width=.45\textwidth]{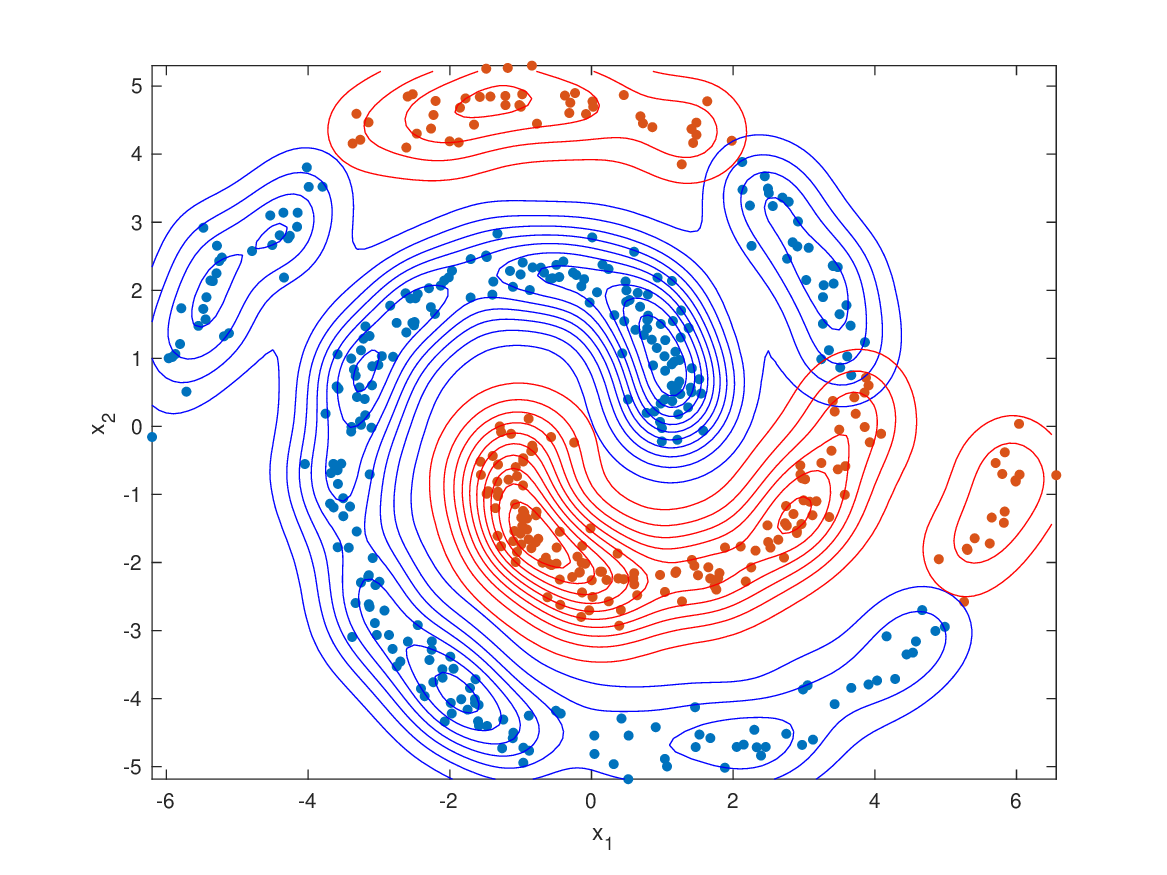} &
\includegraphics[width=.45\textwidth]{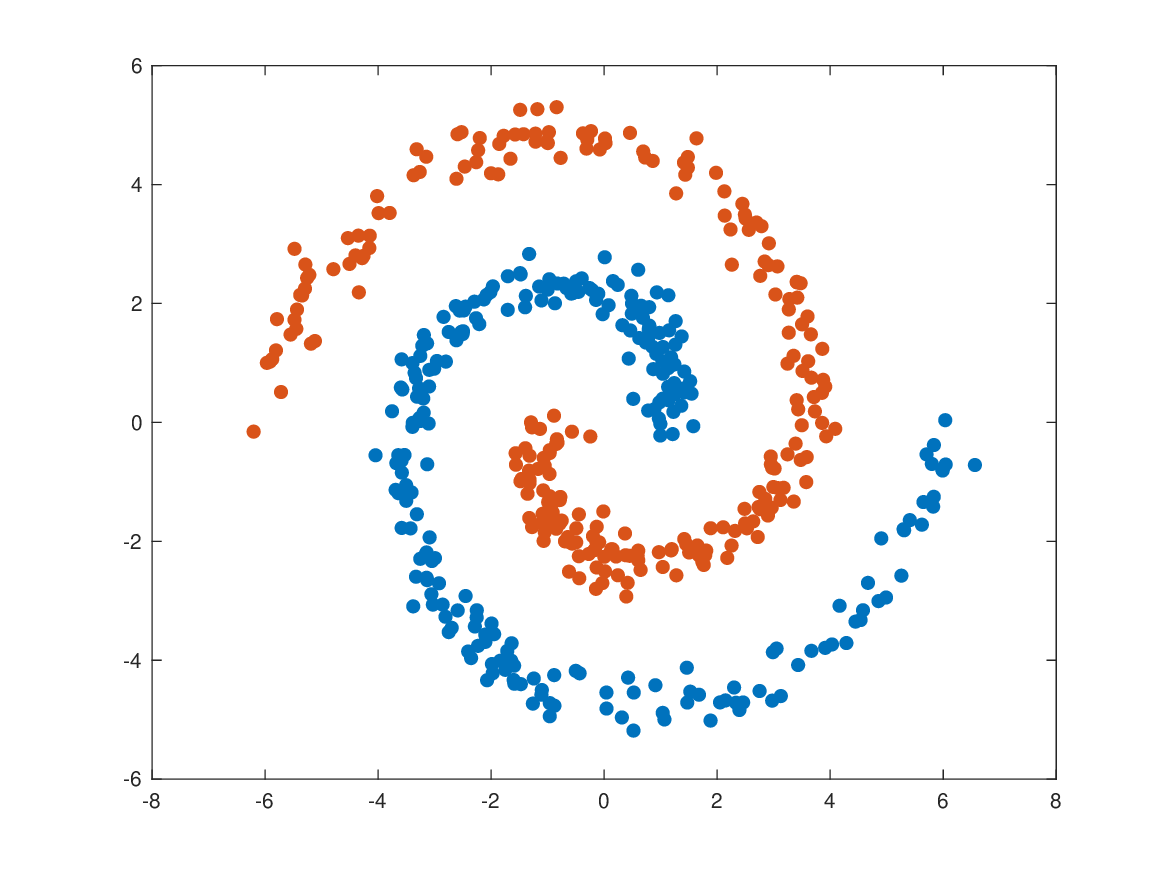} \\
\end{tabular}
\vskip-0.5cm
\caption{Two interlocked spirals. Top left: raw data. Top right: parametric clustering using Gaussians. Middle left: semi-parametric clustering using kernel density estimation, with bandwidth $\sigma = 0.7$. Even though the estimated $\rho_k$ are locally better tuned to the actual data, the procedure still performs a roughly linear partition of the samples.  Middle right: same, with $\sigma = 0.6$. This time the correct clustering is obtained, showing that the true distribution underlying the data is within the reach of kernel density estimation. However, results are sensitive to bandwidth and initialization: the solution landscape is too rich in suboptimal classifications for the algorithm to arrive consistently at the correct one. Bottom left: same, with $\sigma = 0.5$. For smaller bandwidths, each spiral is divided into segments, often joined to nearby segments from the other spiral. Bottom right: class assignment by the non-parametric procedure described herein, not associated to a density estimation per class.}
\label{fig:Spirals}
\end{center}
\end{figure}

While this approach is more flexible, it requires careful choice of parameters: figure \ref{fig:Spirals} shows how, depending on the bandwidth $\sigma$, the algorithm either behaves similarly to its Gaussian counterpart, roughly dividing space into two half-planes (middle-left), or segments the spirals into smaller pieces (bottom-left). In this example, while a precisely tuned value of the bandwidth yields the correct assignments (middle-right), the solution is sensitive to small changes in the bandwidth, the value of $\Delta t$ and the initial, nearly uniform, random soft assignment to classes. The plots also display the corresponding kernel density estimations, which are consistent with the resulting clusters. 

The landscape associated with kernel density estimation is too highly non-convex to give robust clustering results for complex underlying distributions.
 More generally, a procedure grounded on density estimation is not ideal: one can easily conceive situations, such as high-dimensional settings, where robust density estimation is beyond reach, yet clustering still makes sense. Moreover, the data at hand will not always lie in smooth manifolds $X$ where a density can be defined: discrete variables such a gender, day of the week and country of origin are ubiquitous in data bases. With this in mind, the next section introduces a purely non-parametric approach, which bypasses density estimation altogether.

\subsection{Measure driven non-parametric time-evolution} \label{ContEq}

The algorithm of the prior subsection was based on a model for the probability distributions $\rho_k$ within each class, which were Gaussians in the first implementation shown and weighted sums of Gaussian kernels in the second. In order to move beyond these parametric and semi-parametric approaches, consider first an ideal scenario with infinitely many samples $\{x^j\}$ available or, equivalently, with a known probability measure $\rho(x)$ underlying the data. Then the soft assignment $P_k^j(t)$ sought is replaced by a function $P_k(x, t)$, the weights are given by
\begin{equation}
 w_k(x, t) =  \frac{P_k(x, t)}{Z_k(t)}, \quad Z_k(t) = \int P_k(x ,t)\ d\rho(x), 
 \label{Zdef}
\end{equation}
and the conditional probability estimation step yields
\begin{equation}
 \rho_k(x, t) = w_k(x, t) \rho(x)  = \frac{P_k(x, t) \rho(x)}{Z_k(t)}.
 \label{DE}
\end{equation}
Formula (\ref{DE}) can be considered as a natural generalization of the weighted parametric density estimation, where we are weighting the full probability distribution. It is also an instance of Bayes theorem, derivable from the two alternative ways in which one can write the joint probability for $y_k$ and $x$:
$$ \hbox{Prob}\left(y_k, x\right) = \rho(x) P_k(x) = Z_k\ \rho_k(x) . $$
The equation for the posterior then transforms as follows:
$$ \tilde{P}_k^j = \frac{P_k^j \; \rho_k\left(x^j\right)}{\sum_{h=1}^K P_h^j \ \rho_h\left(x^j\right)} \rightarrow \tilde{P}_k(x) =  \frac{P_k(x) \frac{P_k(x) \rho(x)}{Z_k} }{\sum_{h=1}^K P_h(x) \frac{P_h(x) \rho(x)}{Z_h}  }$$

Extending the parametric procedure, one would then evolve $P_k(x,t)$ through the integro-differential equation resulting in the reaction term $R_k(x,t)$ presented in the introduction:
%
\begin{equation}
 \frac{\partial P_k(x, t)}{\partial t} =  R_k(x,t) \equiv \left(\frac{\frac{P_k(x,t)}{Z_k(t)}}{\sum_{h=1}^K \frac{P_h^2(x,t)}{Z_h(t)}}  - 1  \right) P_k(x,t).
 \label{Reaction}
\end{equation}
The steady states of (\ref{Reaction}) and their stability are characterized by lemma \ref{fplemma}, and the ground state by lemma \ref{LLlemma}. These lemmas show that (\ref{Reaction}) necessarily converges to rigid assignments --i.e. to $P_k^j \in \{0, 1\}$-- and that the ground state has classes of equal size $Z_k = \frac{1}{K}$.  

Note that the data (that is, the probability measure $\rho(x)$) appears in (\ref{Reaction}) only through the integrals defining the $Z_k$, which represent the expected values of $P_k$ under $\rho(x)$ or, equivalently, the total mass in each class $k$. This is convenient since, when implementing the non-parametric algorithm in real scenarios, we will need to return to a discrete setting where only samples $\{x^j\}$ of $\rho(x)$ are available. In that case, expected values can be naturally replaced by the corresponding empirical means.
Another consequence though is that the system does not use any metric information, as the algorithm is indifferent to whether two points $x^i, x^j$ are near or far from each other. This is in contrast to the parametric scenarios, where the estimated $\rho_k(x,t)$ linked neighboring points. In order to restore metric information while retaining the fully non-parametric nature of the algorithm, we add diffusion to the evolution equation:
\begin{equation}
  \frac{\partial P_k(x, t)}{\partial t} = R_k(x, t) + \nu D_k(x, t),  \label{ipde}
\end{equation}
where $D$ is a diffusion operator and $\nu$ is a scalar viscosity or diffusivity. In particular, when $X$ is a smooth manifold and $\rho(x)$ is a probability density, we can adopt
\begin{equation}
 D_k = \frac{1}{\rho(x)} \nabla \cdot \left(\rho(x) \nabla P_k(x,t) \right) .
\end{equation}
This relaxes the probabilistic assignments $P_k$ at each point toward their local averages, thus pushing nearby points to have similar assignments. Including the density $\rho$ in the probability flux is natural, as diffusion should only act in the presence of particles to carry the assignment $P_k$. Locations with small density are natural boundaries between clusters, across which one would not want class-assignment to diffuse. The factor $\frac{1}{\rho}$ follows from the requirement that diffusion conserve the total mass $Z_k = \int P_k\ \rho \ dx$ in each class, for which $\rho D_k$ must be in divergence form.

To better understand the nature of the added diffusive terms, rewrite (\ref{ipde}) as an advection-reaction-diffusion equation,
\begin{equation}
 \frac{\partial P_k}{\partial t} -  \nu \frac{\nabla \rho}{\rho} \cdot \nabla P_k =  R_k + \nu \Delta P_k
 \label{ipde_diff}
\end{equation}
and notice that the drift velocity of $P_k$ 
$$ - \nu \frac{\nabla \rho}{\rho}  = -\nu \nabla \ln \rho$$
points in the direction of decreasing \emph{global density}, diverging from areas of high density and converging to areas of low data density, thus creating stronger gradients \emph{between clusters} in these regions. Notice also that, even though (\ref{ipde_diff}) exchanges probability among neighboring values of $x$ for each $k$ independently, it preserves the condition $\sum_k P_k(x, t) = 1$.
%

The system in (\ref{ipde}), together with a dynamic determination of the diffusivity $\nu$ to be discussed below, constitutes the core evolution equation proposed in this article. In this form, it is a novel system of partial integro-differential reaction-diffusion equations, which we will study in more detail in section \ref{sec:analysis}. Our immediate goal, however, is to adapt it to real data situations, where the probability mesure $\rho(x)$ is only known through a set of available sample points.  

\subsection{Sample-driven non-parametric time evolution}

When $\rho(x)$ is only known through a collection of sample points $\{x^j\}$, a natural unstructured grid for the spatial dependence is provided by these observations themselves. The evolution equation, evaluated at these points with $P_k^j(t) =P_k(x^j, t)$, yields the system
\begin{equation}
 \frac{dP_k^j}{dt} = R_k^j + \nu D_k^j ,  \label{Pt2}
\end{equation}
with
\begin{equation}
R_k^j(t) = \left(\frac{\frac{P_k^j}{Z_k}}{\sum_{h=1}^K \frac{{P_h^j}^2}{Z_h}}  - 1  \right) P_k^j,  \qquad  Z_k(t) = \frac{1}{m} \sum_{j=1}^m P_k^j \label{Req}
\end{equation}
and $D_k^j$ a discrete diffusion term, based on the following considerations.

The regular Laplacian is minus the variational derivative of the squared gradient (i.e. the Dirichlet energy), and similarly
$$ \nabla \cdot \left(\rho(x) \nabla P_k(x,t)\right) = 
-\frac{\delta}{\delta P_k} \left[\frac{1}{2} \int \|\nabla P_k \|^2 \rho(x) dx \right]. $$
Thus it is natural to start with a similar measure of the variability of $P_k$ in the graph provided by the observations. Given a distance $d$ in the sample space $X$, we introduce a set of weights $c_{i,j}$ inversely proportional to the square distance between points $x^i$ and $x^j$
$$ c_{i,j} = \frac{1}{d\left(x^i, x^j\right)^2 + \epsilon^2}. $$
When implementing the algorithm we truncate this distance for computational convenience, keeping only the $n$ nearest neighbors, i.e. the $n$ largest values of $c_{i,j}$ for each $i$, setting all other $c_{i,j}$ to zero.

We included an $\epsilon \ll 1$ in the weights in order to mollify them, since pairs of samples $(x^i, x^j)$ can be arbitrarily close. Though such pairs should in most situations be assigned to the same cluster, we would not want their assignments to be rigidly slaved to each other, as a nearly infinite value of $c_{i,j}$ would imply. The reason is that there may be other factors pushing their assignments apart, such as prior information in semi-supervised scenarios, or the existence of additional diffusivity networks, discussed in section \ref{sec:manyfeatures}.\footnote{In order to fix $\epsilon$ in a way that will be consistent with those extensions, we first normalize the $\{x^j\}$ so that they have unit variance,
$ x^j \rightarrow \frac{1}{s} x^j, \quad s = \sqrt{\hbox{tr}\left(x x^t\right)}$,
and then adopt
$ \epsilon = \frac{1}{m}$,
which has the order of a typical distance between neighboring points.}

Now, consider the total weighted variability
$$ C(P) = \frac{1}{2} \sum_{i,j} c_{i,j} \left(P^i-P^j\right)^2. $$
Since
$$ \frac{\partial C}{\partial P^i} = \sum_j \left(c_{i,j} + c_{j,i}\right) \left(P^i-P^j\right),$$
we introduce the graph-Laplacian matrix $L$:
\begin{equation}
 L_i^j =  \left(c_{i,j} + c_{j,i}\right) \quad \hbox{for $i \ne j$}, \quad L_i^i = -\sum_j \left(c_{i,j} + c_{j,i}\right) , 
 \label{gL}
\end{equation}
and define the discrete diffusion as
$$ D_k^i = \sum_j L_i^j P_k^j. $$
Notice that this term does not create probability but just redistributes it, as it preserves the sum of the $P_k^j$ for each class $k$. It also preserves the condition that $\sum_k P_k^j = 1$, like its continuous counterpart. The fact that the matrix $L$ so built is sparse will highly reduce the computational cost of the semi-implicit numerical scheme proposed below.

The argument above was based on translating a variational formulation from the continuous to the sample-based case. If desired, a more direct argument can be made based directly on the final form of the discrete operator. Fixing the class, and
in the simplest case of one-dimensional data with a matrix $c$ built using just the closest neighbor and without mollification, the argument is as follows:
\begin{eqnarray*}
 D^i&=& c_{i, i+1}\left(P^{i+1} - P^i\right) - c_{i-1, i}\left(P^{i} - P^{i-1}\right) \\
 &=& \frac{P^{i+1} - P^i}{\left(x^{i+1} - x^i\right)^2} - \frac{P^{i} - P^{i-1}}{\left(x^{i} - x^{i-1}\right)^2} \\
 &=& \frac{\hat{P}_x^{i+\frac{1}{2}}}{x^{i+1} - x^i} - \frac{\hat{P}_x^{i-\frac{1}{2}}}{x^{i} - x^{i-1}} \quad \left(\hbox{where } \hat{P}_x^{i+\frac{1}{2}} = \frac{P^{i+1} - P^i}{x^{i+1} - x^i} \right)\\
 &=& \frac{1}{x^{i+\frac{1}{2}} - x^{i-\frac{1}{2}}}  \left[\frac{x^{i+\frac{1}{2}} - x^{i-\frac{1}{2}}}{x^{i+1} - x^i} \hat{P}_x^{i+\frac{1}{2}} - \frac{x^{i+\frac{1}{2}} - x^{i-\frac{1}{2}}}{x^{i} - x^{i-1}} \hat{P}_x^{i-\frac{1}{2}}\right] \\
 &=& \frac{1}{x^{i+\frac{1}{2}} - x^{i-\frac{1}{2}}}  \left[\frac{\hat{\rho}^{i+\frac{1}{2}}}{\hat{\rho}^i} \hat{P}_x^{i+\frac{1}{2}} - \frac{\hat{\rho}^{i-\frac{1}{2}}}{\hat{\rho}^i} \hat{P}_x^{i-\frac{1}{2}}\right]  \quad \left(\hbox{where } \hat{\rho}^{i+\frac{1}{2}} = \frac{1}{x^{i+1} - x^i} \right) \\
 &\approx& \frac{1}{\rho^i} \frac{\partial}{\partial x} \left(\rho P_x \right)\Big|_{x^i},
\end{eqnarray*}
where we have used the fact that typical distances between neighboring points are inversely proportional to the local density.

\subsection{Choices for the diffusivity}

In order to cluster data using either the density-driven PDEs in (\ref{ipde}) or the data-driven ODEs in (\ref{Pt2}), one needs to set the value of the diffusivity $\nu$. If this is too small, the influence of the metric is too weak to affect the clustering, and the solution yields spatially unconnected assignments depending solely on the initial assignment of the $P_k^j$. If the diffusivity is too large, on the other hand, the process converges to a uniform, non-informative $P_k^j = \frac{1}{K}$.
 A good value of $\nu$ would establish a balance between the reaction term in (\ref{ipde},\ref{Pt2}) which tends to break the $P_k^j$ into ones and zeros without regards for their proximity, and the diffusive term, which renders spatially uniform the probabilistic assignments to each class.   

To strike this balance automatically and adaptively, we propose
\begin{equation}
  \nu = \alpha \frac{\|R\|}{\|D\|} ,
  \label{alpha}
\end{equation}
where $\alpha$ is a fixed parameter and we are using the Frobenius norm, i.e. the square root of the sum over classes of either the squared Euclidean or $\rho$-weighted $L^2$ norms in the discrete and continuous settings respectively.

To decide on the value of $\alpha$, notice that, at convergence, one must have
$$ R + \nu D = 0, \quad \hbox{so} \quad \|R\| = \nu \|D\|. $$
Hence, if $\alpha < 1$, we necessarily have at convergence $R=0$, which implies that all assignments are rigid, with the $P_k^j \in \{0, 1\}$ (The solution $P_k^j = \frac{1}{K}$, which also makes $R=0$, is unstable.) On the other hand, if $\alpha > 1$, the solution necessarily converges to $D=0$, yielding a uniform assignment for each connected component of the data (here the connectivity is established through the matrix $c$ defining $L$.)
A natural choice is to select a value of $\alpha$ equal to or slightly smaller than $1$, so that diffusivity can bring about the spatial information encoded in the matrix $L$ before the reactive components in $R$ dominate and freeze the solution with hard assignments.

In order to solve (\ref{Pt2}) numerically, we replace the time derivative on the left-hand side with the finite difference approximation
$$ \frac{dP_k^j}{dt} \rightarrow \frac{\left(P_k^j\right)^{n+1} - \left(P_k^j\right)^n}{\Delta t} . $$
For the right hand-side of (\ref{Pt2}), since $R_k^j$ is nonlinear, it is natural to treat it explicitly i.e. to evaluate it at the current time $n$. Doing the same with $D_k^j$, on the other hand, would necessitate the use of very small time intervals $\Delta t$ to avoid the numerical instability associated to the CFL condition for diffusion. Moreover, special care would be required to guarantee that $P$ remains non-negative throughout. Thus we treat the diffusion term implicitly, a choice for which Lemma \ref{TSlemma} shows that any $\Delta t \le 1$ yields a stable, positivity preserving scheme.

Typically we use $\Delta t = 0.99$ and, in clustering experiments, $\alpha = 0.95$. The value of $\alpha$ strictly below $1$ accounts both for the aforementioned desirability of hard assignments and from the fact that the time-discrete problem with a large timestep does not behave identically to the continuous problem. In particular, in order to update the diffusivity $\nu$, we estimate $\|D\|$ explicitly, at time $t^n$, which is slightly inaccurate, since we are applying $D$ implicitly, at time $t^{n+1}$. On the other hand,  we will see below that, in semi-supervised settings, it makes sense to use values of $\alpha$ bigger than one.

The bottom-right panel of figure \ref{fig:Spirals} shows the new non-parametric procedure succeeding on the same example where the parametric and the kernel-density-estimation based approach had mostly failed. Unlike those approaches, the non-parametric procedure does not depend on any probability density estimation.
 
\section{Extensions and examples: semisupervised classification and multiple data types}

The clustering procedure developed above admits a number of quite straightforward extensions. This section describes two such extensions: to semisupervised classification, where some labels are known before-hand, and to the use of complementary features in clustering. As a particularly impactful application of the latter, we develop a new methodology for the clustering of time series, which can be used for determining when a process switches between different regimes.

\subsection{Semisupervised classification}

In semisupervised classification problems, one is given a set of samples for a subset $I$ of which the labels $P_k^i$ are known, and seeks a label assignment for the remaining samples. The procedure extends to this scenario with minimal changes: the known $P_k^I$ are given from the start and not updated by the algorithm, while the others are updated as before. The known samples are included in the Laplacian operators, and therefore the labels of the other samples can diffuse toward the labels of their pre-labeled neighbors. Thus each iteration of the algorithm now reads:

$$ \nu^n = \alpha \frac{\|R^n\|}{\|L P^n\|}, $$
$$    \left(P_k^j\right)^{n+1} = 
\begin{cases}{\left(P_k^j\right)^{n} + \Delta t \left(R_k^j\right)^n + \nu^n \Delta t \sum_i L_i^j \left(P_k^i\right)^{n+1} 
  & \hbox{for $j \not\in I$}, \cr
   \left(P_k^j\right)^{n} & \hbox{for $j \in I$} }\end{cases} $$

Unlike pure clustering, in semisupervised classification one can use values of $\alpha$ larger than one for the determination of the diffusivity, since the existence of available labels forbids a uniform $P$. As a consequence of this higher diffusivity, not all $P_k^i$ necessary converge to either $0$ or $1$. Figure \ref{fig:SemiSupervised} shows that intermediate values of $P$ do indeed show up in areas where the class assignment is not unambiguously determined by the available data.

\begin{figure}[!htpb]
\begin{center}
\begin{tabular}{lll}
\includegraphics[width=.49\textwidth]{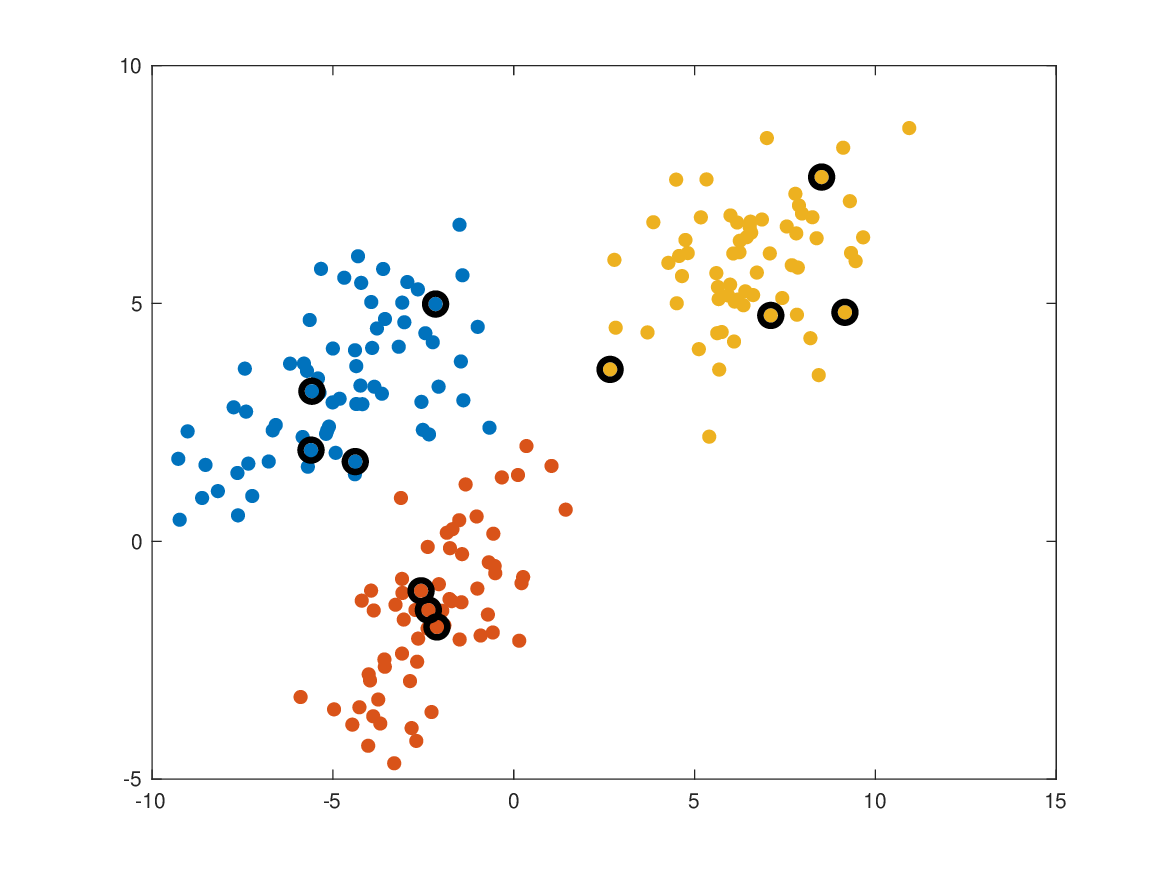} &
\includegraphics[width=.49\textwidth]{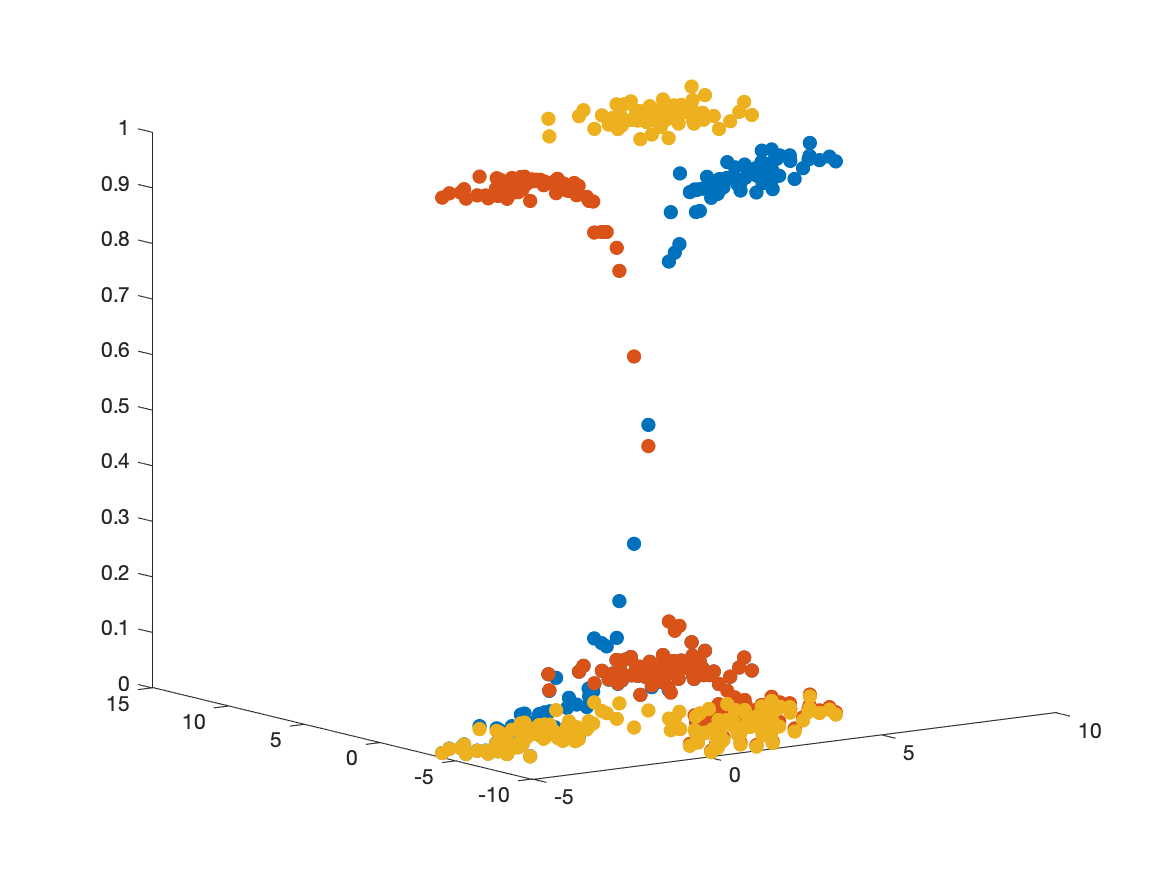} \\
\end{tabular}
\caption{Semi-supervised clustering. Left: data-points, with their assignments displayed by color and black circles marking those points where the labels were known. Right: probabilities $P_k^j$. Notice that while the points in the yellow cluster are assigned unambiguously, those points near the boundary between the red and blue clusters receive soft assignments, as permitted by a value $\alpha = 1.75 > 1$ used for the adaptive diffusivity.
}
\label{fig:SemiSupervised}
\end{center}
\end{figure}

\subsection{Multiple metrics (independent diffusions)}
\label{sec:manyfeatures}

One can think of the $c_{i,j}$ as weights attached to a connectivity graph based on the geometry underlying the samples $\{x^i\}$. A natural extension applies to situations with more that one such ``geometry'' exists. These various metrics may be associated to different features or covariates, to real networks associated to the data points or, as discussed in the next subsection, to space and time. 

The procedural extension is quite straightforward: one replaces (\ref{Pt2}) with  
\begin{equation}
 \frac{dP_k^j}{dt} = \left(\frac{\frac{P_k^j}{Z_k(t)}}{\sum_{h=1}^K \frac{{P_h^j}^2}{Z_h(t)}}  - 1  \right) P_k^j + \sum_l \nu_l \sum_i L(l)_i^j P_k^i , \label{MultiDiffusion}
\end{equation}
where each value of $l$ corresponds to a distinct connectivity network and each operator $L(l)$ takes into consideration only the distances in the $l$'th metric. Thus class assignments can diffuse between distant points, provided that they are close in some of the metrics.

In order to extend the adaptive determination of the diffusivity, one sets beforehand a set of weights $\lambda_l > 0, \ \sum_l \lambda_l = 1,$ attached to each network, so that $\nu_l = \nu \lambda_l,$ with
\begin{equation}
  \nu = \alpha \frac{\|R\|}{\sum_l \lambda_l \|D_l\|} ,
  \label{alpha_multi}
\end{equation}
and $ D_l = L_l P. $


\subsubsection{Clustering of time series}

A natural application of the extension just discussed is to the clustering of time series, where the goal is to classify behavior into different regimes and find the corresponding switching times. Here one associates to each observation $x^j$ the time $t_j$ as a covariate, and defines two sets of connectivity weights: one (or more) associated with the space of the $\{x^j\}$, and another with proximity in time. As a consequence, probability diffuses in time as well as space, inhibiting frequent jumps among clusters in the time variable.

For a simple example, consider two clusters, with overlapping distributions
$$ \rho_{\pm}(x) = N\left(\pm 1, 1\right), $$
and a process whereby the cluster is determined by the sign of a latent variable $z$ satisfying the SDE
\begin{equation}
  dz = -V_z \ dt + dW, \quad V = z^4 - 2 z^2, 
  \label{dz}
\end{equation}
with two meta-stable states symmetric at $z = \pm 1$. We draw 500 samples of $x$  at regular intervals $\Delta t = 0.02$. The data for the clustering algorithm consists of pairs $(t^j, x^j)$, each to be assigned to one of the two clusters.
Figure \ref{fig:Time_series} displays the generation of data for this example and the reconstruction by the algorithm of the time evolution of the underlying two-regimes. 

\begin{figure}[htpb]
\begin{center}
\begin{tabular}{ll}
\includegraphics[width=.49\textwidth]{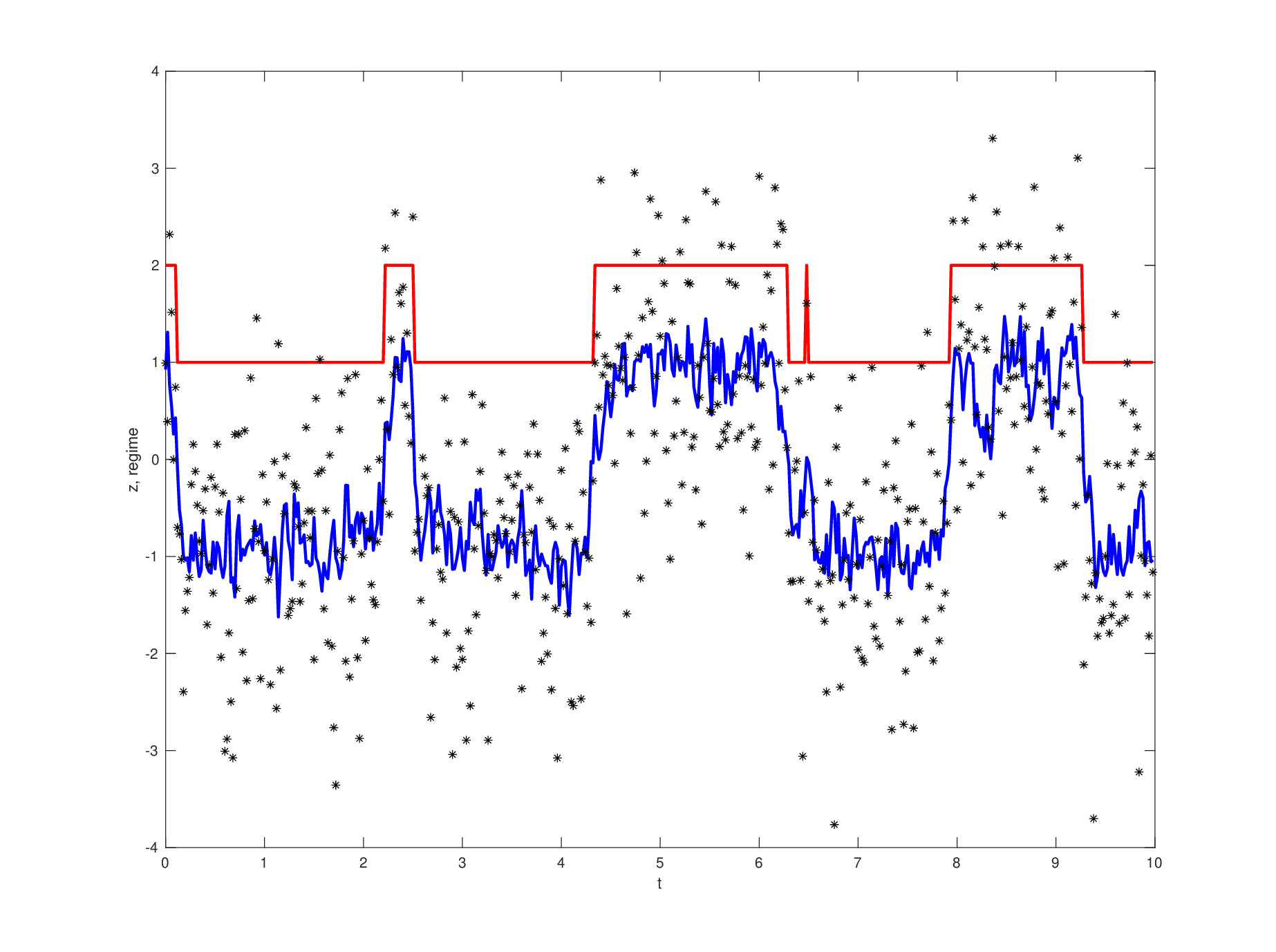} &
\includegraphics[width=.49\textwidth]{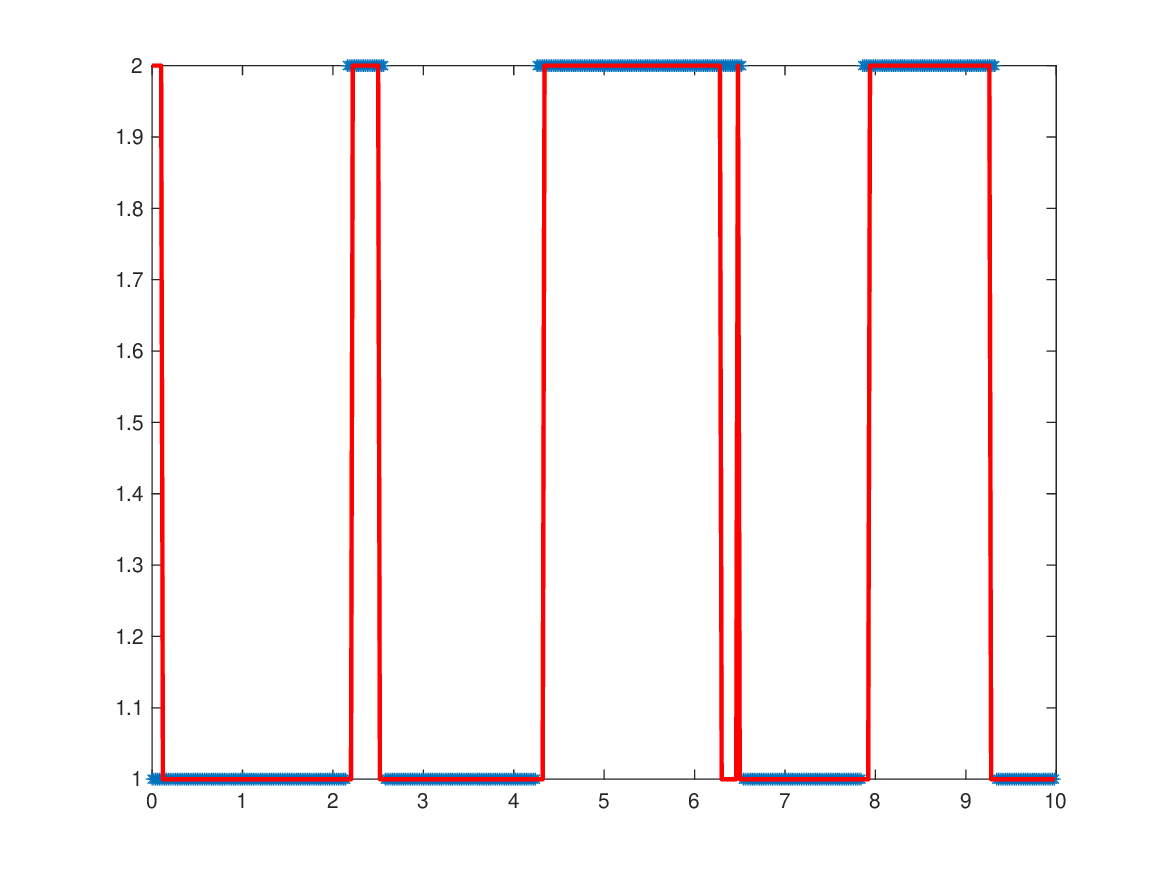} \\
\end{tabular}
\caption{Clustering of a time series. The plot on the left displays the data-generation process: a sample trajectory $z(t^j)$ of a discretized version of the stochastic process in (\ref{dz}) in blue, the corresponding assignment to the classes  $c^j \in \{+, -\}$ (plotted in red as $1$s and $2$s) and the samples $x^j$ from the corresponding $\rho_{\pm}$ as black stars. The plot on the right displays in blue the classes reconstructed by the algorithm, in excellent agreement with the actual classes (in red).
}
\label{fig:Time_series}
\end{center}
\end{figure}

\subsubsection{A synthetic medical example}

Diffusion through multiple metrics can also be applied to the diagnosis of medical conditions with a variety of symptoms which are not always present. We provide here an example where conditions are reconstructed from symptoms. Let the variables $x_l$ for $l={1, \ldots, L}$ represent measurable quantities where the symptoms may manifest. For instance, we may have $x_1 = \hbox{body temperature}$, $x_2 = \hbox{blood pressure}$, $x_3 = \hbox{white cell count}$, etc. Let $\rho_b^l(x_l)$ represent the background, asymptomatic distributions for these variables. The classes $k = 1,\ldots,  K$ represent the various possible conditions to diagnose, such as the common cold, the flu, pneumonia or no ailment at all. Denote by $p_k$ the prior probability of condition $k$ in the population under study.

For each condition $k$ and variable $x_l$, there is a probability $s_k^l \in (0, 1)$ that the symptoms manifest, and a distribution $\rho_k^l(x_l)$ when they do. It follows that the full distribution $R^l(x_l)$ for each variable $x_l$ is given by
$$ R^l(x_l) = \sum_k p_k\left(s_k^l \rho_k^l(x_l) + \left(1-s_k^l\right) \rho_b^l(x_l)\right).$$

In a synthetic example, we generate data $\{x_l^j\}$ for $n$ patients as follows. For each patient $j$, we draw their condition $k$ from the distribution $p_k$. Then, for each variable $x_l$, we draw the presence of symptoms from $s_k^l$ and then draw $x_l^j$ from either $\rho_b^l(x_l)$ or $\rho_k^l(x_l)$ accordingly. In the example shown below, we have adopted $L=8$ real variables $x_l$,  $K=3$ conditions with probability $p_k = 1/3$ each, $n=500$ patients, background conditions
$$\forall l \ \rho_b^l = N(0, 1), $$
probabilities of symptom manifestation
$$ \forall k\ \forall l \ s_k^l = \frac{3}{4}, $$
and symptomatic variable distributions
$$ \rho_k^l = N\left(\mu_k^l, 1 \right), \quad \mu_k^l = \frac{1}{2} +\left(3 k - 2 l\right). $$

We applied the algorithm with one diffusive network per variable, with equal weight $\lambda_l = \frac{1}{8}$ assigned to each. Figure \ref{fig:Clinical} displays the results of the converged assignment in four plots, each corresponding to the plane of two variables, with the samples colored by condition. The five misdiagnosed cases, circled in black, lie either in the asymptomatic area around $x_l = 0$ or at the interface between the distribution corresponding to two conditions. By contrast, running  on the same data the code with one single eight-dimensional diffusive network yielded a much greater number of misclassifications, 76 in this case.

\begin{figure}[!htpb]
\begin{center}
\includegraphics[width=\textwidth]{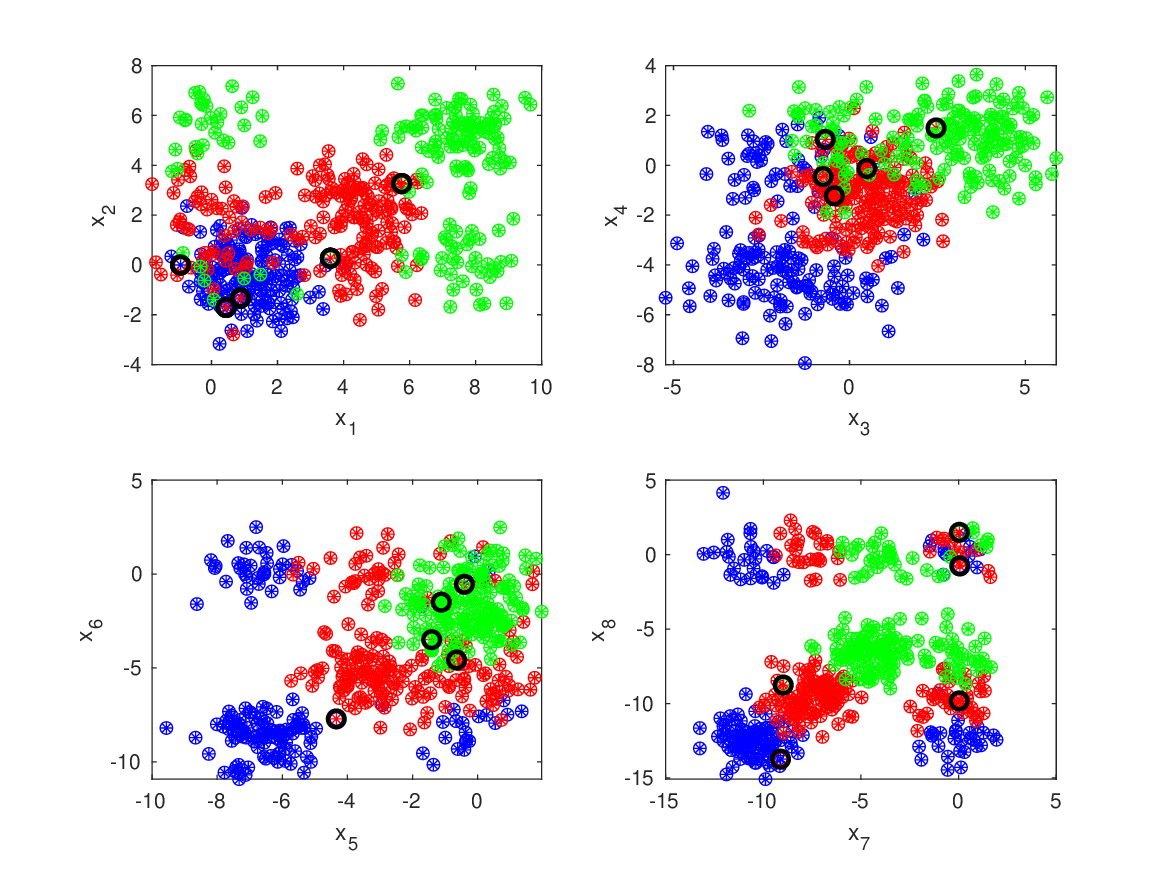}
\caption{Results from a synthetic medical example with eight clinical variables associated to distinct diffusive networks, clustered into three possible diagnoses, displayed by color. The five misclassifications, marked by black circles, display few symptoms, which are moreover either ambiguous or more typically attributable to different conditions. 
}
\label{fig:Clinical}
\end{center}
\end{figure}

Besides its immediate medical applicability, this example points to general features of the potential of clustering in a multi-metric setting:

\begin{enumerate}
\item {\bf The blessing of dimensionality.} One can repeat the synthetic medical example changing the number $L$ of variables observed --and corresponding networks. Figure \ref{fig:Clinical2} displays the number of misclassifications as a function of $L$, a number that decays rapidly, reaching in this typical realization of the data, one for $L>8$ and
 zero for $L>15$. On the one hand, this is consistent with the intuitive fact that each additional observable $x_l^j$ provides additional information on the class $k^j$. However, it runs contrary to the ``curse of dimensionality'', which haunts clustering algorithms such as $k$-means in high dimensions. Assigning a network per variable avoids this curse altogether, as outlying values of a particular variable $x_l$ do not stop the probability $p_k$ from diffusing across the other available networks.

\item {\bf Robustness to outliers.} More generally, the availability of more than one diffusive network makes the clustering algorithm robust to outliers, since high barriers along some diffusive network do not inhibit diffusion along the others. 
\end{enumerate}

\begin{figure}[!htpb]
\begin{center}
\includegraphics[width=.75\textwidth]{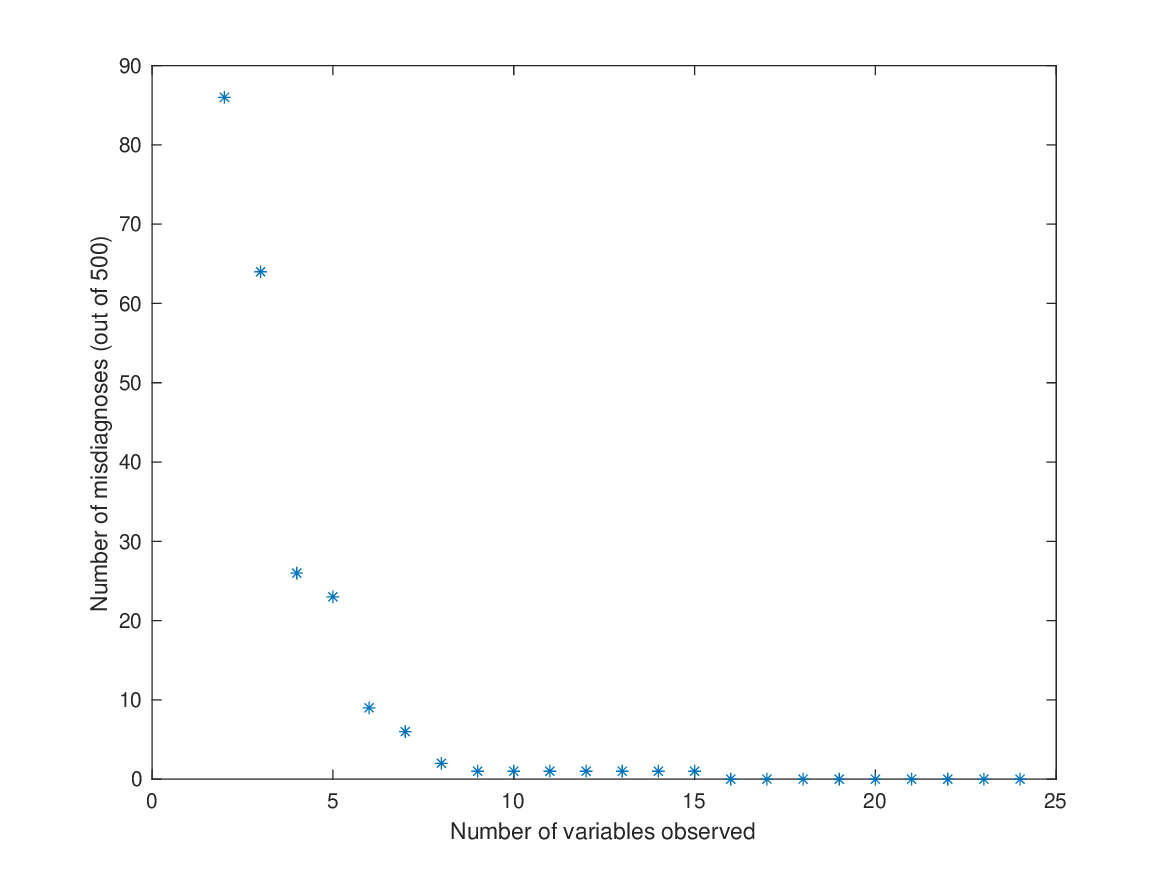}
\caption{Number of miss-classifications, for a fixed data set, as a function of the number of variables observed. When an individual network is assigned to each variable, high-dimensionality is a blessing, not a curse.
}
\label{fig:Clinical2}
\end{center}
\end{figure}

\section{Properties of the reaction--diffusion system}

Some properties of the reaction--diffusion system of Section \ref{ContEq} are proved in this section that shed light on its applicability and on the numerical scheme that we use to solve it. First we study what steady solutions the system has and which are dynamically reachable, also interpreting the evolution in terms of maximising a likelihood. Then we prove that the proposed numerical method is stable for time-steps $\Delta t \le 1$. Finally, we derive the continuous limit of multi-diffusive networks, which gives rise to a family of non-local diffusion operators.

\subsection{Convergence}

We first show that the only stable steady solutions to (\ref{Reaction}) are rigid assignments with no empty class.

\begin{lemma} \label{fplemma}

All steady states of (\ref{Reaction}) can be built by selecting $K$ numbers $Z_k \ge 0$ with
$$ \sum_k Z_k = 1. $$
and a subset $W$ of the space $X$ (the set of rigid assignments) with measure $w$ (extreme choices have $W$  empty or $W = X$). Then

\begin{enumerate}

\item For each $x \in W$, select a class $k$ and assign $P_k(x) =1$ and $P_{h \ne k}(x) = 0$, so that
$$ \forall k \ \int_{W} P_k(x) \ d\rho(x) = w\ Z_k. $$

\item For all $x \not\in W$, assign
$$ P_k(x) = Z_k. $$

\end{enumerate}
Of these steady states, only those with $W = X$ --i.e. only rigid assignments-- and no class with $Z_k = 0$ are dynamically stable.

\end{lemma}

\begin{proof}

Steady states of (\ref{Reaction}) must satisfy
$$
 \forall x \ \forall k \ \left(\frac{\frac{P_k(x)}{Z_k}}{\sum_{h=1}^K \frac{P_h^2(x)}{Z_h}}  - 1  \right) P_k(x) = 0. 
$$
Then, for each value of $x$, either $P_k(x) = 0$ or 
\begin{equation}
\forall k\ \frac{P_k(x)}{Z_k} = \sum_{h=1}^K \frac{P_h^2(x)}{Z_h} .
 \label{second-case}
\end{equation}
Since the right-hand side of (\ref{second-case}) is independent of $k$ and both the $P_k$ and the $Z_k$ add up to one, we must have, in this case,
$$ P_k(x) = Z_k , $$
as stated in 2. 

On the other hand, when all but one of the $P_h(x)$ vanish, the remaining one with $P_k = 1$, satisfies (\ref{second-case}) automatically, since the sum on its right-hand side has only one non-zero term. That the relative measure of the rigid assignments in $W$ must agree with their corresponding $Z_k$ follows from the definition of $Z_k$ in (\ref{Zdef}), completing the statement in 1.

The instability of the steady non-rigid assignments, where $0<P_k(x) <1$ and $P_k(x) = Z_k$, can be seen as follows. Fix one such point $x$, and consider a perturbation 
$$ P_k = Z_k + \epsilon_k(t), \quad \epsilon_k \ll 1, \quad \sum_k \epsilon_k = 0. $$
Then
$$  \frac{d \epsilon_k}{d t} = \left(\frac{\frac{Z_k + \epsilon_k}{Z_k + \mu \epsilon_k}}{\sum_{h=1}^K \frac{\left(Z_h + \epsilon_h\right)^2}{Z_h + \mu \epsilon_h}}  - 1  \right) \left(Z_k + \epsilon_k\right), $$
where $\mu$ is the measure of the point $x$ (zero if no measure is assigned to small sets and $\frac{1}{N}$ for a discrete, equidistributed measure on $N$ points). To leading order in the $\{\epsilon_h\}$, we have
$$  \frac{d \epsilon_k}{d t} \approx \left(\frac{1 + (1-\mu) \epsilon_k}{1 + O\left(\|\epsilon\|^2\right)}  - 1  \right) \left(Z_k + \epsilon_k\right) \approx Z_k (1-\mu) \epsilon_k $$
Since $\mu < 1$ and $Z_k > 0$, the soft probability assignment is unstable.

To show that hard assignments to classes with $Z_k \ne 0$ are stable, choose an $x$ with $P_k(x) = 1$ and, for any $h \ne k$, consider the perturbation
$$ P_h = \epsilon_h(t). $$
Then
$$  \frac{d \epsilon_h}{d t} = \left(\frac{\frac{\epsilon_h}{Z_h + \mu \epsilon_h}}{\sum_{l=1}^K \frac{\left(P_l + \epsilon_l\right)^2}{Z_l + \mu \epsilon_l}}  - 1  \right) \epsilon_h
\approx \left(\frac{Z_k}{Z_h} \epsilon_h  - 1  \right) \epsilon_h \approx -\epsilon_h, $$
which shows stability.

Finally, in order to show that a class $k$ with $Z_k = 0$ is unstable, consider a point $x$ with measure $\mu$  rigidly assigned to another class $h$, and perturb the corresponding $P_k$ away from $0$, i.e. consider $P_k = \epsilon_k(t)$. We have
$$  \frac{d \epsilon_k}{d t} \approx \left(\frac{\frac{\epsilon_k}{\mu \epsilon_k}}{\frac{1}{Z_h}}  - 1  \right) \epsilon_k
= \left(\frac{Z_h}{\mu} - 1\right) \epsilon_k. $$
If point $x$ is strictly contained in class $h$, then $Z_h > \mu$ and $P_k = 0$ is unstable.

\end{proof}

Notice that, in the discrete case, not all real values of $Z_k$ are possible steady states:
assume there are $0 \le n \le m$ points $x_j$ with probability $P_k^j \in \{0,1\}$, that is, $n$ points have hard assignments to a particular cluster. Denoting the number assigned to cluster $k$ (i.e $P_k^j=1$) by $n_k \ge 0$ (with $\sum_k n_k = n$), then, when $n_k \neq 0$ 
$$Z_k = n_k/m.$$
For any remaining labels with $n_k=0$, one may choose $Z_k$ such that $\sum_k Z_k = 1$.  The remaining $m-n$ points (the soft assignments) will then have
$$P_k^j = Z_k.$$


In order to further study the flow of the system (\ref{Reaction}) and interpret the solutions towards which it converges, we introduce a Lyapunov function, the log-likelihood of the data points, that strictly increases as the system evolves. 

To develop an intuition for this Lyapunov function, consider first the parametric version of the problem in section \ref{ParamEq}. It solves a maximum likelihood problem:
$$ \max_{P, \alpha} L = \sum_j \log\left(\sum_{k=1}^K P_k^j \rho_k\left(x_j; \alpha_k\right) \right).$$
There are $K$ probability distributions $\rho_k$ depending on parameters $\alpha_k$, and each sample point $x_j$ is softly assigned to each of the $K$ classes with probability $P_k^j$. The $P$ and $\alpha$ are determined so as to maximize the likelihood function $L$. By contrast, a density estimation by the mixture of $K$ distributions has a similar likelihood $L$ and parameters $\alpha$, but the individualized assignments $P_k^j$ are replaced by global probabilities $P_k$ (our $Z_k$), the weight of each distribution in the mixture.

Switching now to the non-parametric setting, consider first the idealized situation where the distribution $\rho(x)$ is known and one seeks the assignment $P_k(x)$. Then the $\rho_k(x)$ are given by (\ref{DE}), and the likelihood function
becomes
$$ L = \int \log\left(\sum_k \frac{P_k(x)^2 \rho(x)}{\int P_k(y) \rho(y) \ dy} \right) \rho(x) \ dx, $$
while the diffusivity adds a non-smoothness cost
$$ C =  \frac{\nu}{2} \sum_k \int \|\nabla P_k(x)\|^2 \rho(x) \ dx , $$
that we will consider separately.

When the known distribution $\rho(x)$ is replaced by samples $\{x_j\}$, the likelihood function becomes
\begin{equation}
 L = \sum_j \log\left(\sum_k \frac{{P_k^j}^2}{Z_k} \right), \quad Z_k = \frac{1}{N} \sum_j P_k^j,
 \label{LD}
\end{equation}
and the cost for non-smoothness becomes
$$ C = \frac{\nu}{2} \sum_{i,j,k} c_i^j \left(P_k^i-P_k^j\right)^2 . $$

\begin{lemma} \label{LLlemma}

The reaction component of our clustering algorithm,
\begin{equation}
 \frac{d P_k^j(t)}{d t} = \left(\frac{\frac{P_k^j(t)}{Z_k(t)}}{\sum_{h=1}^K \frac{{P_h^j}^2(t)}{Z_h(t)}}  - 1  \right) P_k^j(t), \quad Z_h(t) = \frac{1}{N} \sum_j P_h^j(t),
 \label{Reaction2}
\end{equation}
converges to a local maximizer of the likelihood $L$ in (\ref{LD}). At the global maximum or ground state, the corresponding classes have equal size, i.e. all $Z_k$'s are equal.

\end{lemma}

\begin{proof}
Consider first the direction of maximal ascent of $L$ when all $P_k^j$ are strictly positive. Bringing in Lagrange multipliers $\lambda_j$ to enforce the condition that $\sum_k P_k^j = 1$, we have
$$ L \rightarrow  \sum_j \left[ \log\left(\sum_k \frac{{P_k^j}^2}{Z_k} \right) + \lambda_j \left(\sum_k P_k^j - 1 \right) \right], $$
so
$$ \frac{\partial L}{\partial P_k^j} = 2 \frac{\frac{P_k^j}{Z_k}}{\sum_h \frac{{P_k^j}^2}{Z_h}} + \lambda_j,$$
where we have temporarily considered the $Z_k$ as fixed external parameters. If $\nabla L \ne 0$, any
$$ dP_k^j = \alpha_k^j  \frac{\partial L}{\partial P_k^j}, \quad \alpha_k^j > 0 $$
is a direction of ascent of $L$, since
$$ dL = \sum_{j, k}  \frac{\partial L}{\partial P_k^j} dP_k^j = \sum_{j, k} \alpha_k^j \left(\frac{\partial L}{\partial P_k^j}\right)^2 > 0 . $$
In particular, we can adopt $\alpha_k^j = P_k^j$, which yields
$$ dP_k^j = 2 \frac{\frac{{P_k^j}^2}{Z_k}}{\sum_h \frac{{P_k^j}^2}{Z_h}} + \lambda_j P_k^j, $$
where the condition that $\sum_k P_k^j = 1$ (and so $\sum_k dP_k^j = 0$) requires that
$$ \lambda_j = - 2, $$
so 
$$ dP_k^j = 2 \left(\frac{\frac{P_k^j}{Z_k}}{\sum_h \frac{{P_k^j}^2}{Z_h}} -1 \right) P_k^j
= 2 \frac{d P_k^j(t)}{d t}, $$
with $\frac{d P_k^j(t)}{d t}$ as defined in (\ref{Reaction2}). It follows that the dynamics in (\ref{Reaction2}) ascends $L$ when the $Z_k$ are fixed external parameters and all $P_k^j$ are strictly positive (When some reach zero, the condition that $P_k^j \ge 0$ becomes active, which again agrees with the dynamics in (\ref{Reaction2}), which freezes zero values of $P_k^j$.)

We have shown in lemma \ref{fplemma} that this dynamics converges to rigid assignments $P_k^j = \delta_k^{k_j}$, where $k_j$ is the class $k$ assigned to observation $j$. Then
$$ L = \sum_j \log\left(\sum_k \frac{{P_k^j}^2}{Z_k} \right) \rightarrow 
- \sum_j \log\left(Z_{k_j}\right) = - N \sum_k \tilde{Z}_k \log\left(Z_{k}\right), $$
where
$$ \tilde{Z}_k = \frac{1}{N} \sum_j P_k^j $$
is the actual size of class $k$ under the current assignment. But this expression for $L$, a relative entropy, achieves its maximum when $Z_k = \tilde{Z}_k$. Thus maximizing $L$ over the external parameters $Z_k$ yields their correct definition in terms of the $P_k^j$.

Finally, maximizing the resulting entropy
$$ L = - N \sum_k Z_k \log\left(Z_{k}\right) $$
over $Z_k$, yields equipartition:
$$ Z_k = \frac{1}{K}, $$
proving that the ground state of (\ref{Reaction2}) has a rigid assignment to classes with equal size.

\end{proof}

Putting together the results of lemmas \ref{fplemma} and \ref{LLlemma}, we conclude that the reactive component of the algorithm leads to rigid assignments where no class is left empty, with a tendency toward equally sized classes, which constitute the ground state. When we add the diffusive component, the objective function $L$ is augmented by 
$$ C \propto \sum_k \int \|\nabla P_k(x)\|^2 \rho(x) \ dx , $$
which penalizes jumps between classes, favoring interfaces between them that are small in size and located in areas of small density $\rho$.

\subsection{Time stepping}

A key to an efficient implementation of the methodology is that one can adopt time-steps $\Delta t$ of order one without compromising the algorithm's stability. This is proved in the following lemma.

\begin{lemma}\label{TSlemma}
Replacing the system of ODEs in (\ref{Pt2}) by the semi-implicit Euler scheme
\begin{equation}
 \frac{\left(P_k^j\right)^{n+1} - \left(P_k^j\right)^n}{\Delta t}  =  \left(R_k^j\right)^n + \nu \left(D_k^j\right)^{n+1}, \label{Euler}
\end{equation} 
with 
\begin{equation}
   \Delta t \le 1,
\label{dt}
\end{equation}
preserves the properties that $P_k^j \ge 0$ and $\sum_k P_k^j = 1$. 
Here the $\left(\cdot\right)^n$ refer to evaluating the corresponding terms at time-step n.

\end{lemma}

\begin{proof}

The following two properties follow immediately from the definitions of $R$ and $D$: 
\begin{enumerate}

\item If, for all $j$, $\sum_k P_k^j = 1$, then $\sum_k R_k^j = \sum_k D_k^j = 0$. (P1)

\item If all $P_k^j$ are non-negative, $R_k^j$ is bounded below by $-P_k^j$. (P2)
\end{enumerate}
In view of P1, the sum over classes $\sum_k P_k^j$ remains always equal to one.
P2, on the other hand, jointly with the constraint (\ref{dt}) for $\Delta t$,  guarantee that the $P_k^j$ remain in the interval $[0,1]$. To see this,
consider the index $j$ where $(P_k)^{n+1}$ achieves its minimal value. We have
\begin{eqnarray*}
 \left(P_k^j\right)^{n+1} &=& \left(P_k^j\right)^n + \Delta t  \left(R_k^j\right)^n + \nu \Delta t \left(D_k^j\right)^{n+1} \\
 &\ge& (1 - \Delta t) \left(P_k^j\right)^n + \nu \Delta t \left(D_k^j\right)^{n+1} \\
 &\ge& 0,
\end{eqnarray*}
since $\Delta t \le 1$ by design, $\left(P_k^j\right)^n \ge 0$ by inductive hypothesis, and 
\begin{eqnarray*}
 \left(D_k^j\right)^{n+1} &=&
\sum_i L_j^i \left(P_k^i\right)^{n+1} 
= \sum_{i \ne j} L_j^i \left(P_k^i\right)^{n+1} + L_{jj} \left(P_k^j\right)^{n+1} \\
&=& \sum_{i \ne j} L_j^i \left(P_k^i - P_k^j\right)^{n+1} \ge 0
\end{eqnarray*}
by the minimality of $\left( P_k^j\right)^{n+1}$ and the non-negativity of the off-diagonal entries of $L$. Thus all $P_k^j$ remain non-negative. Since $P_k^j$ add to one for each $j$, it follows that none can exceed one either.

\end{proof}

\subsection{The continuum limit of multi-diffusions} \label{MultiDiff}

Recall from section \ref{sec:manyfeatures} the multi-diffusive operator
\begin{equation}
 D_d P_k = \sum_l \nu_l \sum_i L(l)_i^j P_k^i,
 \label{DM}
\end{equation}
where each $L(l)$ acts over a subset of the coordinates of $x$.
In the continuous case, one might naively think that the discrete operator
should converge, when each network is associated with a feature $x_l$, to the non-isotropic diffusion
$$
 \sum_l \nu_l \nabla_l \cdot \left(\rho(x_l) \nabla_l P_k(x, t)\right),
$$
where $\nabla_l$ is the gradient operator acting along the $l$th feature and
$$ \rho^l(x_l) = \int \rho(x) \prod_{h\ne l} dx_h $$
is the $x_l$-marginal of $\rho(x)$.
Yet the true extension of (\ref{MultiDiffusion}) to the continuous scenario must involve a nonlocal process: points that are far in one connectivity network, even not connected at all, may still diffuse in another. Regular diffusive processes can never achieve this.

In order to extend (\ref{DM}) to a continuous setting, we consider two families of averaging operators:
$$ \overline{u}^{\,l}\left(x_l\right) = \frac{\int u(x) \rho(x) \prod_{h\ne l} dx_h}{\int \rho(x) \prod_{h\ne l} dx_h}, $$
which averages $u(x)$ along all directions but $x_l$, and a local averaging operator over $x_l$, such as
$$ \widehat{v}^{\,l}\left(x\right) = \int K_a\left(s-x_l\right) v\left(x_1, \ldots, x_{l-1}, s, x_{l+1}, \ldots x_L\right) \ ds, $$
where $K_a$ is a kernel function with mass 1, and  bandwidth $a$ that approaches zero. This  kernel  defines a diffusion operator along the $x_l$-direction:
\begin{equation}
 D_l v = 2 \frac{\widehat{v}^{\, l}\left(x\right) - v\left(x\right)}{a^2}. 
 \label{av_diff}
\end{equation}
For instance, adopting
$$ K_a(s-z) = \frac{\delta(s - z + a) + \delta(s - z - a)}{2} $$
yields, in the limit $a\rightarrow 0$,
$$ D_l v = \frac{v\left(\ldots, x_l + a, \ldots \right) - 2 v\left(\ldots, x_l, \ldots \right) + v\left( \ldots, x_l - a, \ldots \right)}{a^2} \rightarrow \frac{\partial^2 v}{\partial {x_l}^2}. $$
More generally, any local averaging operator $\widehat{.}$ defines a diffusion operator along $x_l$ through (\ref{av_diff}). Then, the continuous limit of (\ref{DM}) becomes
$$ D_d P_k =  \sum_l \nu_l \left( \widehat{\overline{P_k}}^{\, l} - P_k \right).$$
This is an interesting sum of non-local diffusion operators: the $l$-th operator subtracts $P_k$ not from its local average, but from an average that, for the $l$-th diffusive network, is local only in $x_l$, but global over all other variables. In two dimensions, $D_d$ can be visualized as a diffusion based on averaging over a thin cross of width $a$ with the point $(x, y)$ at its center.

\section{Simulations in the continuum limit}
\label{sec:analysis}

Combining a gradual, nonparametric Bayesian upgrade of the soft probability assignments $P_k(x,t)$ with a diffusive process to make the assignments consistent with the metric of the space $X$ gave rise to the system (\ref{ipde}). This set of $K$  reaction diffusion equations is non-conventional, as it includes a non-local element, embodied in the integrals defining the $Z_k(t)$, and has an adaptive diffusivity that automatically balances the reactive and diffusive components of the evolution at all times. Experimenting with this system in the continuum limit, in addition to providing insight into the mechanics of the clustering procedure in the limit of infinitely many samples, is interesting per se, as it turns out to display interesting dynamics.

We consider for simplicity the case with two classes $1$ and $2$. Then $P_2 = 1 - P_1$,  $Z_2 = 1 - Z_1$, and we have
$$ \sum_{k=1}^2 \frac{{P_k}^2}{Z_k} = \frac{{P_1}^2 -2 Z_1 P_1 + Z_1}{Z_1 (1-Z_1)}$$
so the system in (\ref{ipde}) reduces to a single integro-differential equation for $P(x, t)=P_1(x, t)$:
\begin{equation}
 \frac{\partial P}{\partial t} = R + \nu D,
\label{TwoClasses}
\end{equation}
where
$$ R = \left(\frac{P(1-Z)}{P^2 -2 Z P + Z}-1\right) P
= \frac{P(1-P)(P-Z)}{(P-Z)^2 + Z(1-Z)}, $$
with
\begin{equation}
 Z(t) = \int P(x, t) \rho(x)\ dx,
 \label{Z_int}
\end{equation}
and
$$ D = \frac{1}{\rho(x)}\nabla\left(\rho(x) \nabla P\right). $$
The diffusive evolves adaptively according to
$$ \nu = \alpha \sqrt{\frac{\int R^2 dx}{\int D^2 dx}}. $$

Without the coupling of $Z$ and $P$ through (\ref{Z_int}), equation (\ref{TwoClasses}) would be a regular reaction diffusion equation, with three fixed critical points: the two stable ones $P=1$ and $P=0$, separated by the unstable $P=Z$. Thus the solution would rapidly evolve into regions of uniform $P=1$ or $P=0$, separated by a diffusive interface (when $\nu$ is constant). With $Z$ smaller that $\frac{1}{2}$, the potential well at $P=1$ would be deeper than the one at $P=0$, so diffusion would act across the interface, displacing it until $P=1$ everywhere (assuming that $\rho(x)$ is nowhere $0$) Similarly, if $Z > \frac{1}{2}$, then $P$ will converge to a uniform $P=0$. 

Yet in our system, $Z$ is the expected value of $P$, evolving as $P$ does. This modifies the drift process described above, shifting its uniform final state to a steady configuration with well-balanced classes. If at some point the mass with $P=1$ is larger than the one with $P=0$, then $Z > \frac{1}{2}$, and the drift towards $P=0$ makes $Z$ decrease. We first illustrate the dynamics with one-dimensional simulations. In the simplest example with constant $\rho$, this leads to two equal clusters and $Z=\frac{1}{2}$, as shown in figure \ref{fig:Cont1d1} (top). When $\rho(x)$ is not uniform, the clusters may accommodate their sizes so that their interfaces lie in areas of smaller density, as seen in the example in figure \ref{fig:Cont1d1} (middle and bottom). If the underlying $\rho(x)$ has two well separated ``clusters'' then the algorithm will choose boundaries to respect the natural conditional probability, i.e. the ratio of $Z$ to $1-Z$ will converge to the ratio of the mass in each cluster (middle). If the underlying probability density has two peaks, the algorithm will converge to a cluster that bisects the density (bottom).

\begin{figure}[!htpb]
\begin{center}
\includegraphics[width=.65\textwidth]{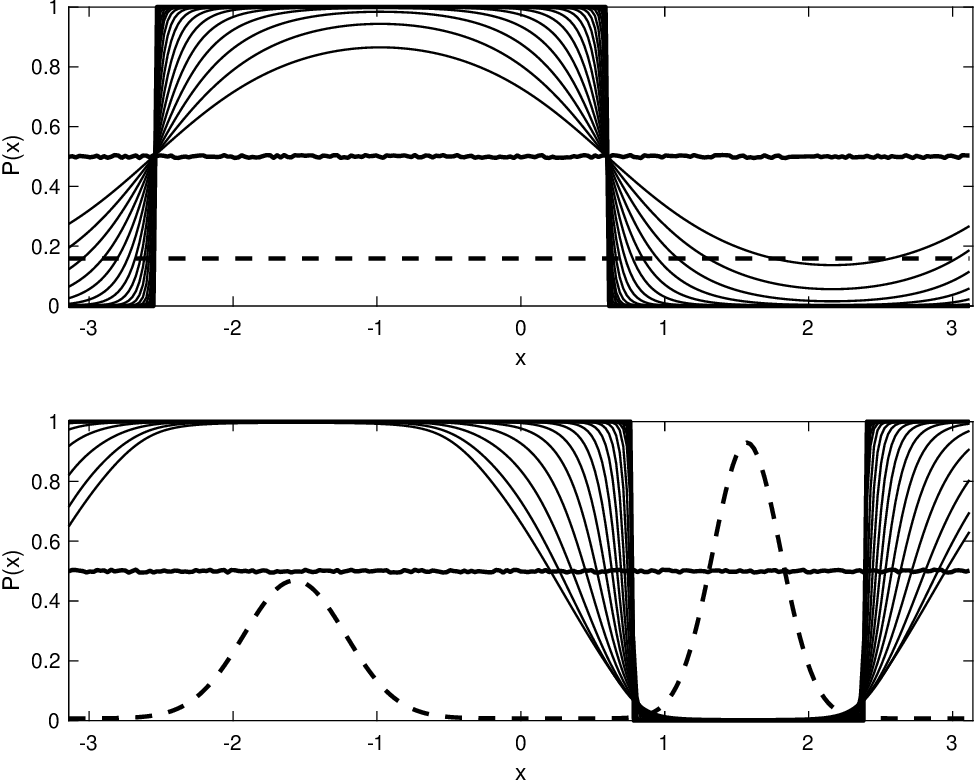}

\includegraphics[width=.65\textwidth]{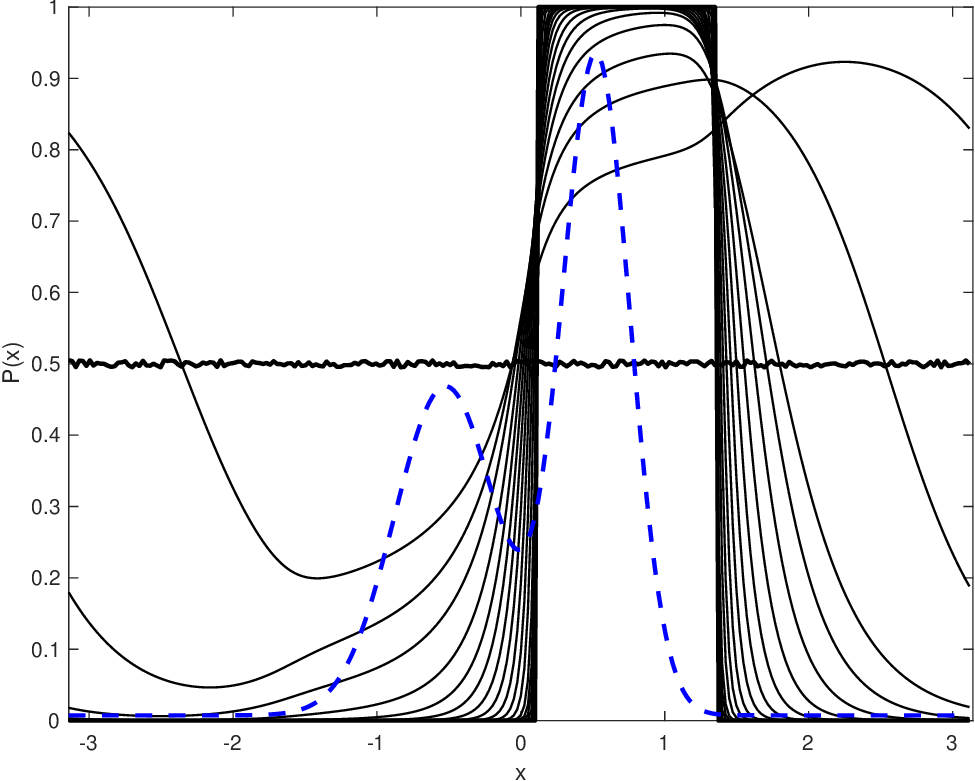} 
\caption{Continuous Dynamics on $S^1$. In all figures the dashed line is the underlying probability $\rho(x)$, the other lines correspond to $P(x)$ at different times, from the initial line with $P(x)$ initialized to $0.5$ plus noise, to the final assignment where $P(x) \in \{0,1\}$. Top: underlying probability $\rho(x)$ is uniform, and algorithm converges to equipartitioned clusters with arbitrary boundaries (depending on initial conditions). Middle: underlying probability is given by 2 well separated Gaussians with different masses. Bottom: underlying probability is given by 2 mixed Gaussians with different masses.}
\label{fig:Cont1d1}
\end{center}
\end{figure}

The simulations shown in figure \ref{fig:Cont1d1} used  an implicit time-stepping for the diffusion and an explicit one for the reaction. The timestep was chosen to be 1, with  256 spatial points in $[0,2\pi)$ and periodic boundaries, reflecting the dynamics on $S^1$. In figure \ref{fig:Cont1d1} we chose $\alpha = 0.95$. The effect of changing $\alpha$ to values away from $1$ is shown in figure \ref{fig:Cont1d2}. Here, the same computation as in figure \ref{fig:Cont1d1} (middle) was repeated with a lower value of $\alpha$ where nonlinearity selects arbitrary clusters (depending on the randomised initial conditions) and a larger value of $\alpha$ where diffusion dominates and clusters are not identified with $P$ converging to $1/2$.

\begin{figure}[!htpb]
\begin{center}

\includegraphics[width=.65\textwidth]{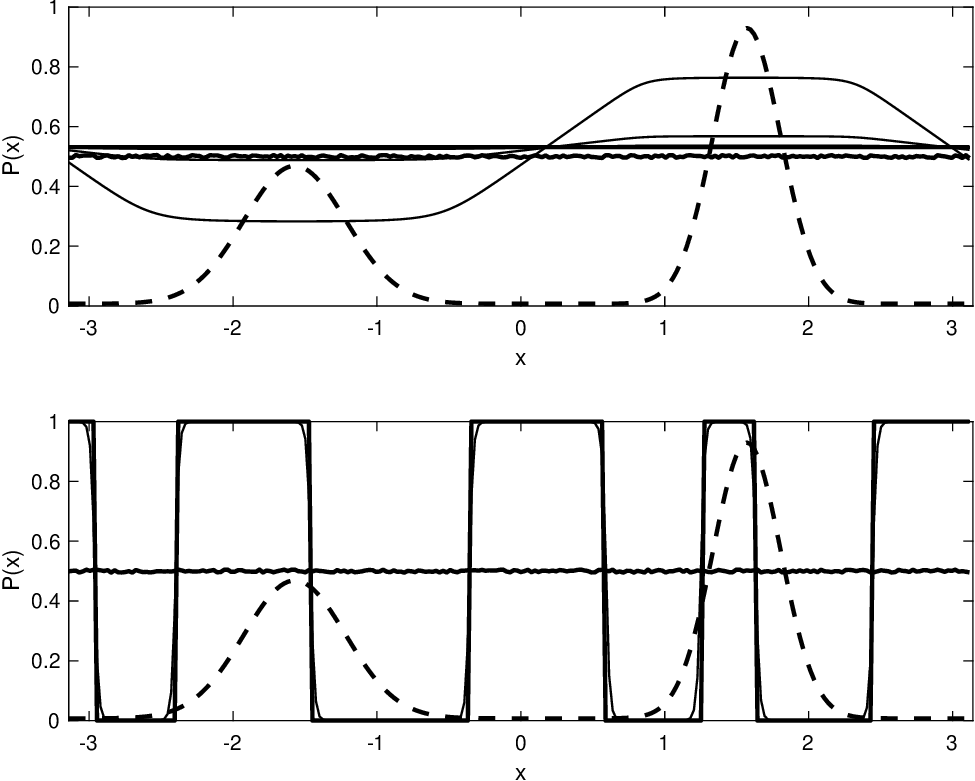}

\caption{The effect of unbalanced diffusion and nonlinearity. Top: $\alpha = 1.2$ where diffusion is too strong to result in clustering, despite the solution tending to cluster at intermediate times. Bottom: $\alpha = 0.5$ where nonlinearity selects arbitrary clusters dependent on the (random) initial data.
}
\label{fig:Cont1d2}
\end{center}
\end{figure}

Next, we consider simulations in two dimensions with results of a typical simulation shown in figure \ref{fig:2d}. One observes initially a coarsening dynamics, as expected in reaction diffusion problems, followed by the separation into sharply defined regions that follow the intuitive classification of the features of the underlying distribution. The low probability background is assigned primarily to one of the clusters. This simulation used a $256\times256$ cartesian grid with a time-step of $0.1$ and $\alpha=1$. A smaller time-step was used in order for the variable diffusivity to more accurately reflect the balance of nonlinearity and diffusion. In parallel to the previous case, too low values of $\alpha$ freeze in sharp domain boundaries too soon, whereas too large values of $\alpha$ again lead to uninformative uniform $P$.

\begin{figure}[!htpb]
\begin{center}
\begin{tabular}{lll}
\includegraphics[width=.45\textwidth]{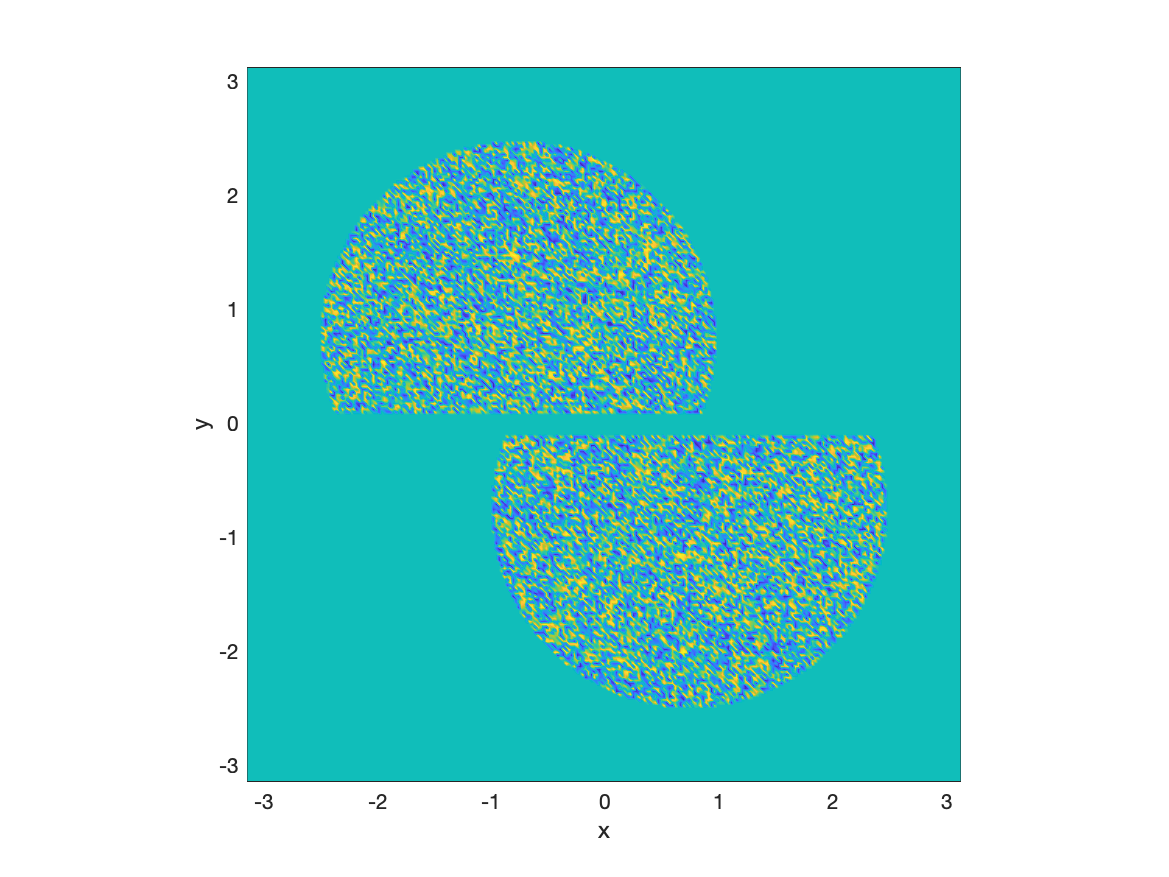} &
\includegraphics[width=.45\textwidth]{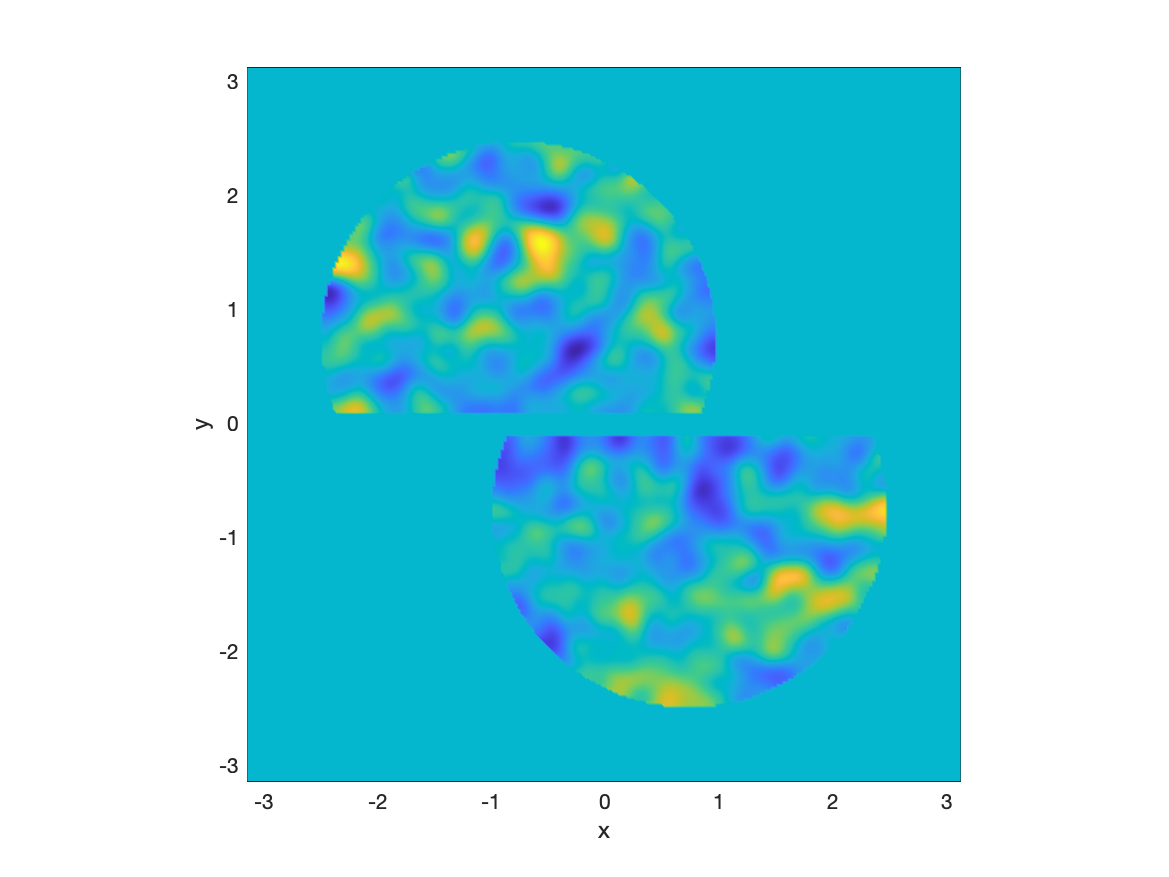} \\
\includegraphics[width=.45\textwidth]{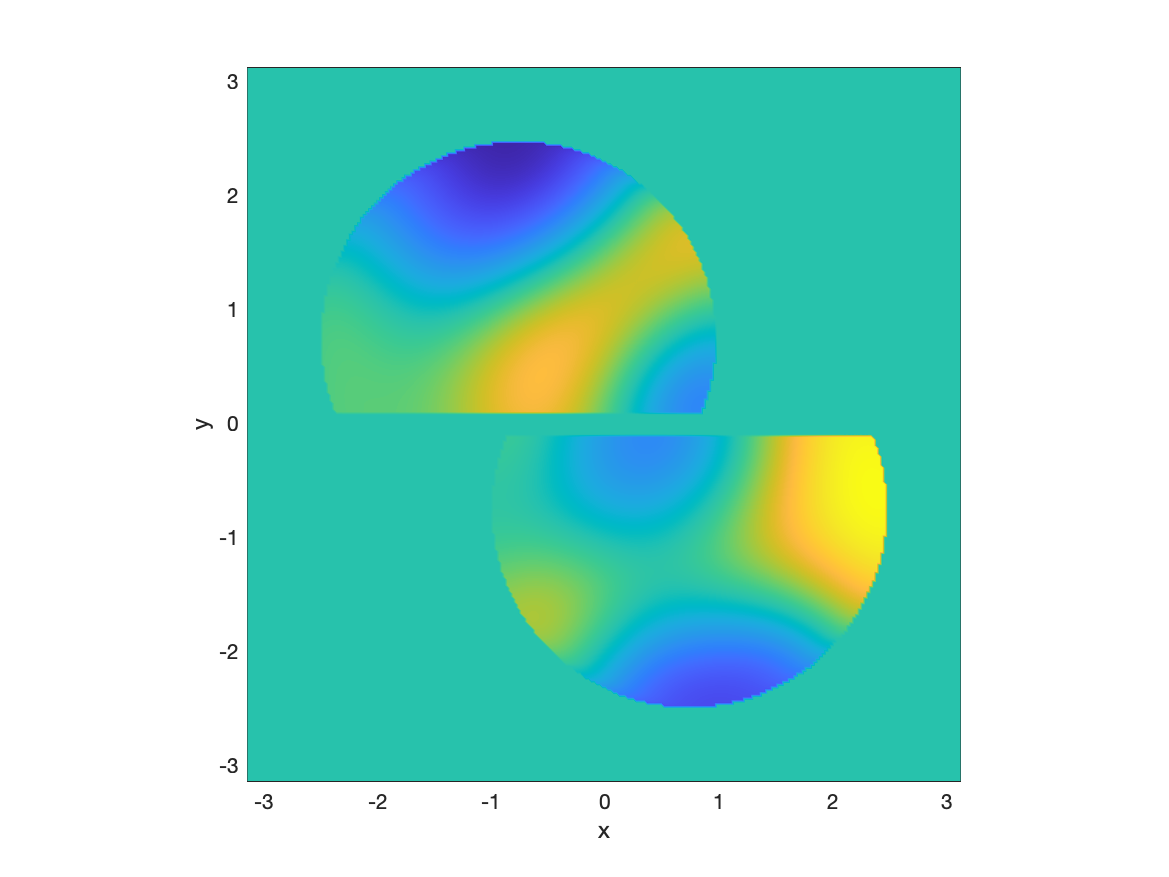} &
\includegraphics[width=.45\textwidth]{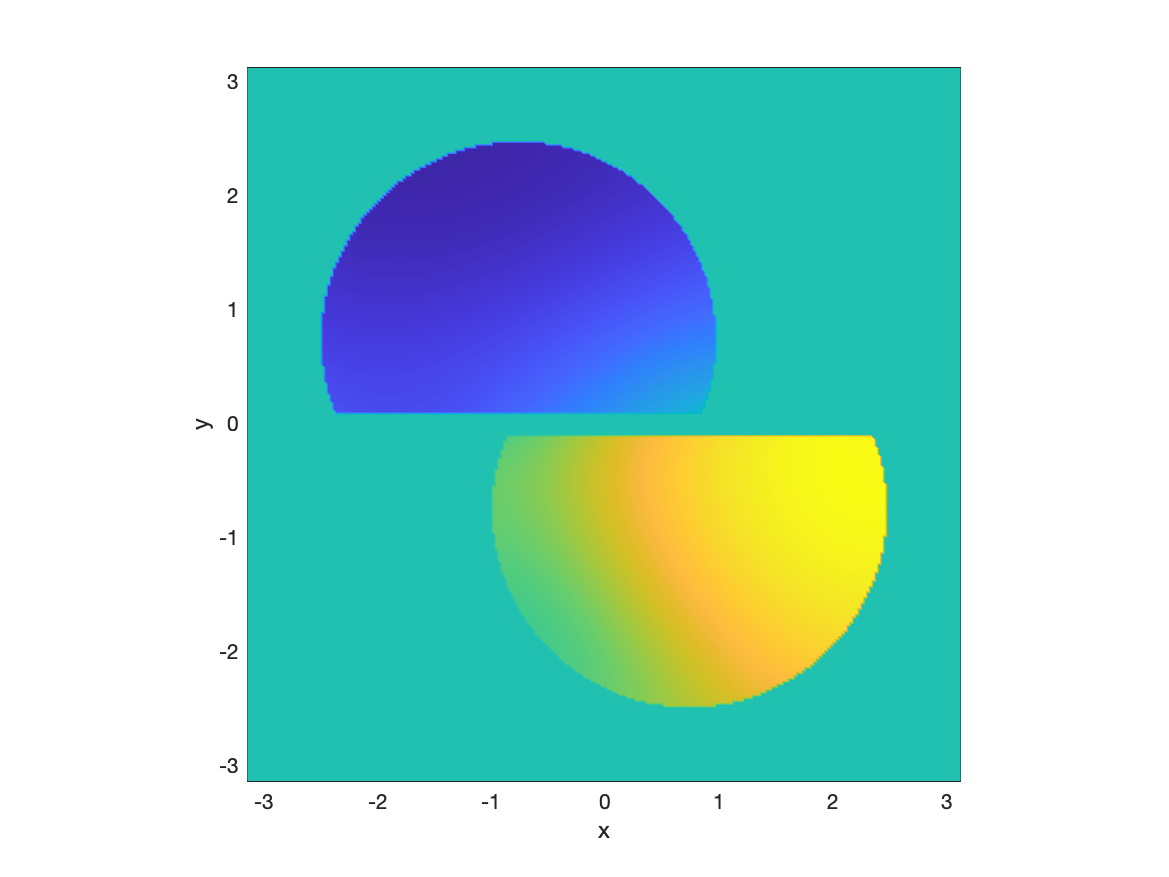} \\
\includegraphics[width=.45\textwidth]{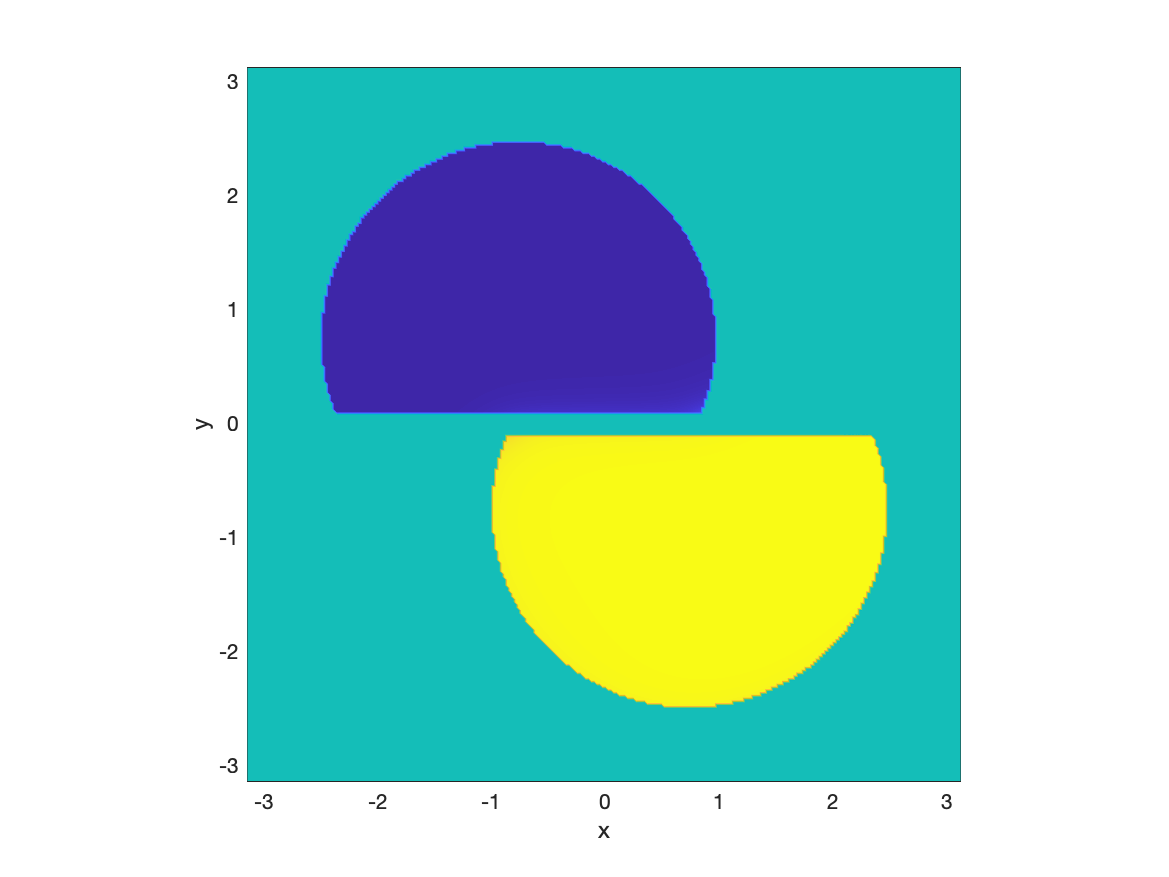} &
\includegraphics[width=.45\textwidth]{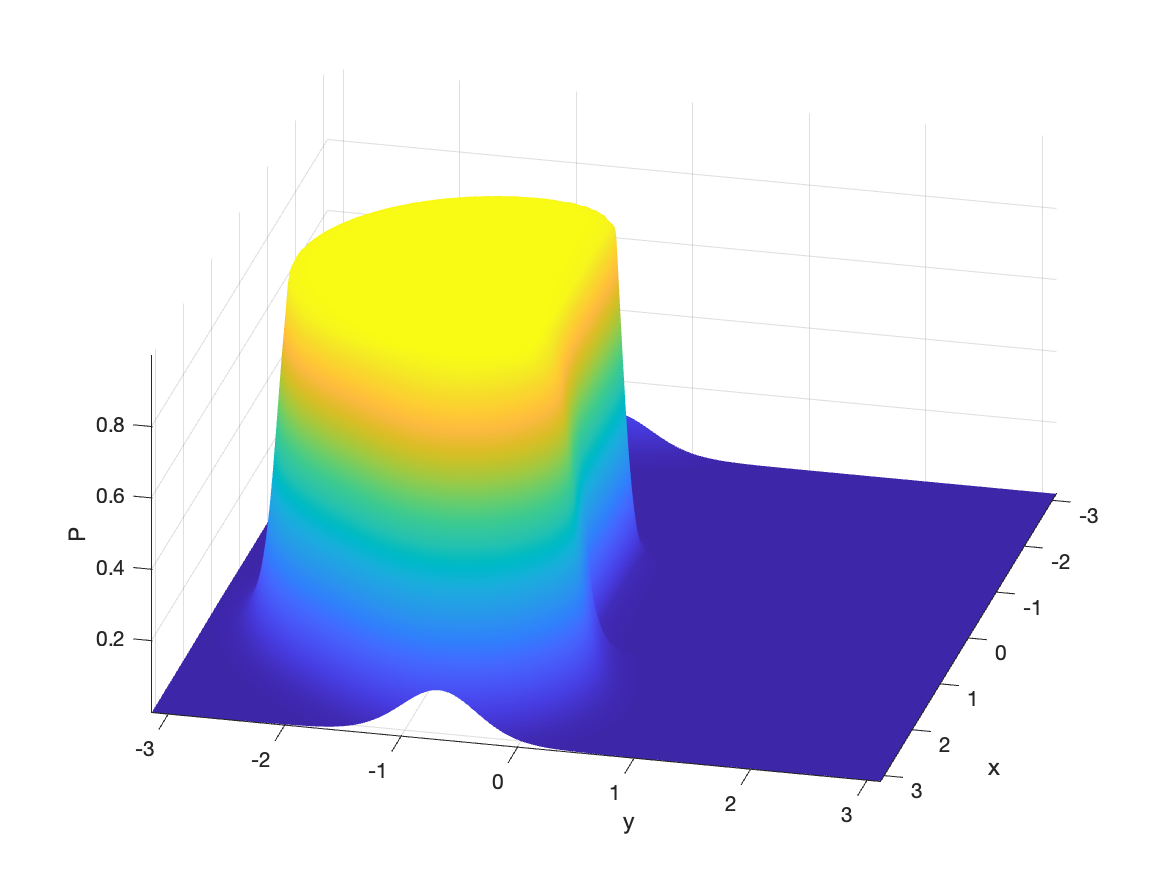} \\
\end{tabular}
\vskip-0.5cm
\caption{Two dimensional runs. The underlying distribution $\rho(x,y)$ is composed of three constants: $\rho(x,y)=0.0445$ in the upper truncated disc, $\rho(x,y)=0.085$ in the lower truncated disc, and $\rho(x,y)=0.004$ in the remaining part of the domain. The figures show $P(x,y,t)$ restricted (artificially) for clarity on the truncated discs at $t=0,3,6,15,21$ from top left to bottom left. The (unrestricted) function $P(x,y,21)$ is shown at bottom-right. In the colour scale blue corresponds to $P=1$ and yellow to $P=0$ and therefore assignment to the 2 clusters.}
\label{fig:2d}
\end{center}
\end{figure}

Finally we consider the case discussed in section \ref{MultiDiff} where diffusions in the $x$ and $y$ directions are independent and nonlocal. The results are shown in figure  \ref{fig:multidiff}, where the initial data is similar to other experiments, small random fluctuations about $P=1/2$. The diffusion, which is strong at initial times, results in a characteristic ``tartan'' pattern as the diffusion operator averages over crosses centred at points in the domain. At later times the evolution converges to clusters which include observations that may be distant in one of the variables so long as they are close in another. In this particular case the two Gaussians which have mass near $y=0$ are clustered together despite being distant in $x$. Conversely, the two Gaussians that are closest end up in different clusters because they do not share similarity in neither $x$  nor $y$.

\begin{figure}[!htpb]
\begin{center}
\begin{tabular}{lll}
\includegraphics[width=.45\textwidth]{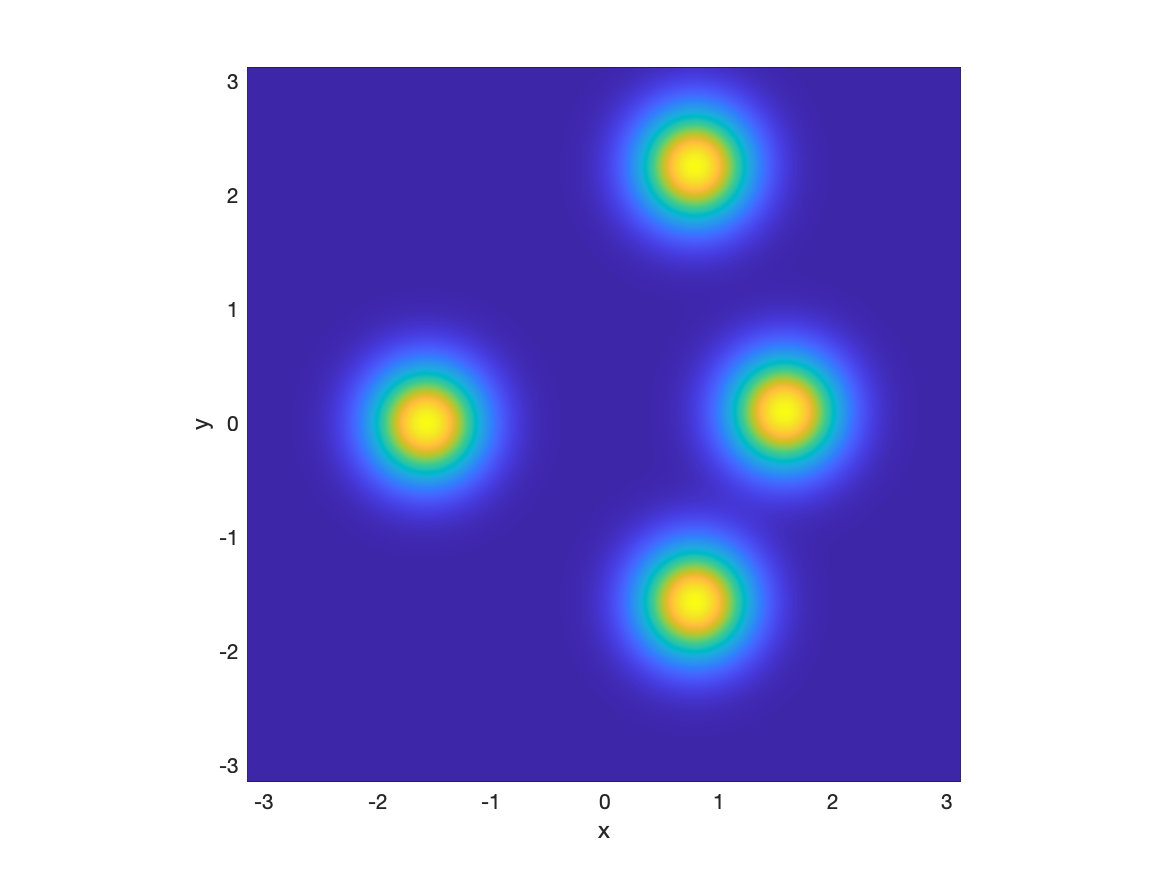} &
\includegraphics[width=.45\textwidth]{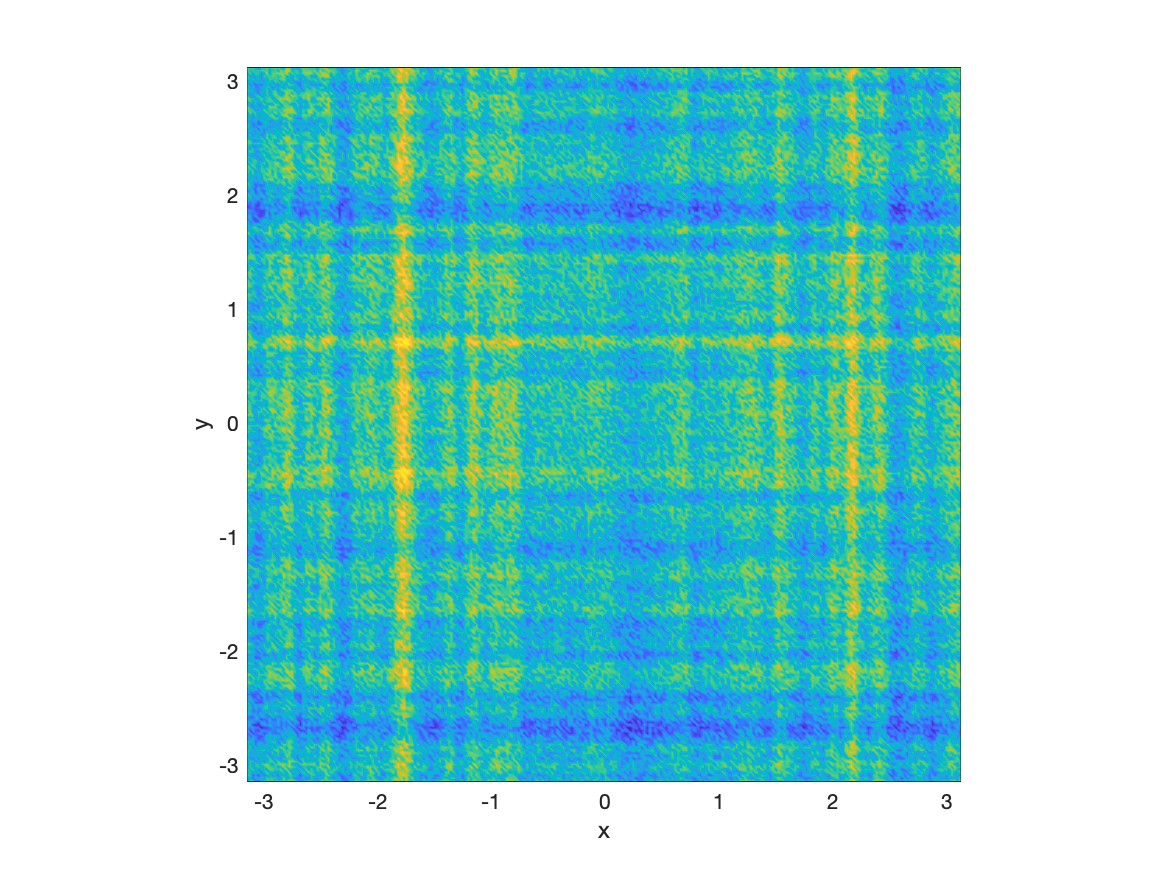} \\
\includegraphics[width=.45\textwidth]{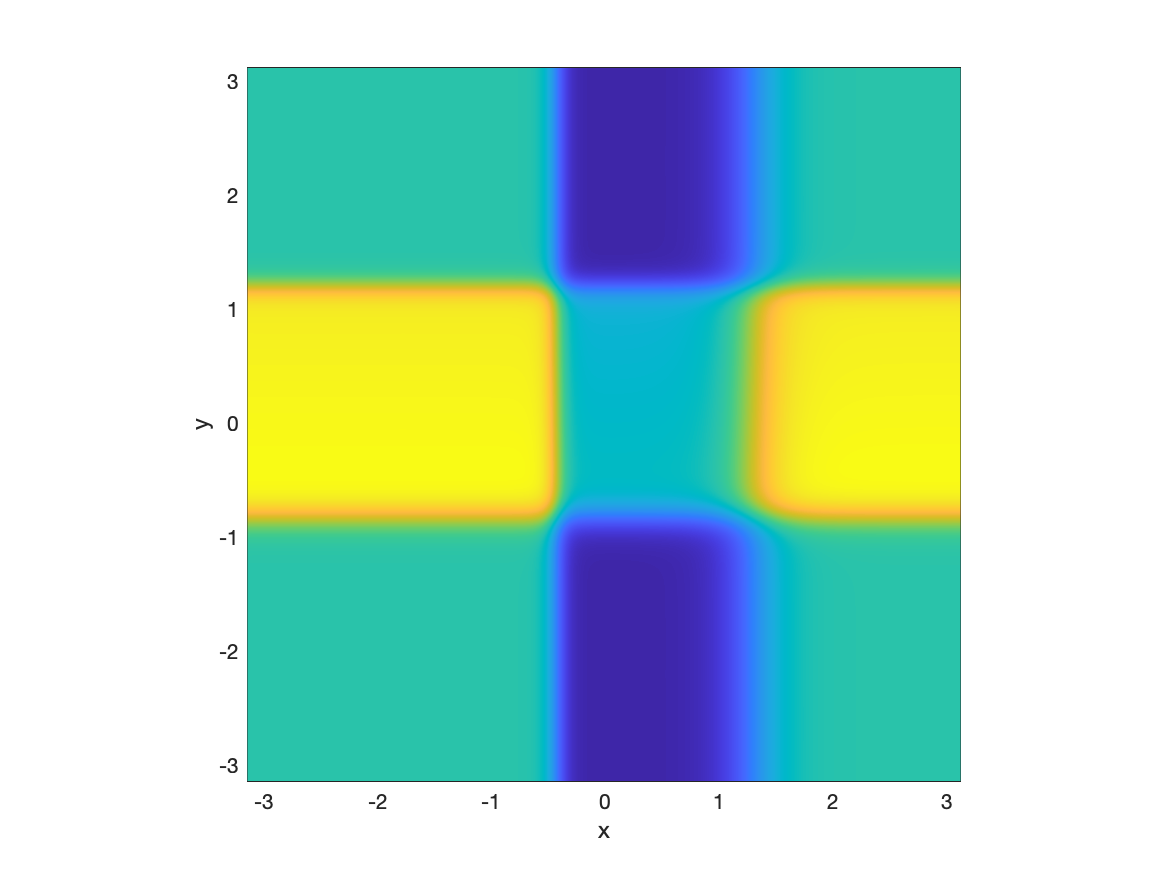} &
\includegraphics[width=.45\textwidth]{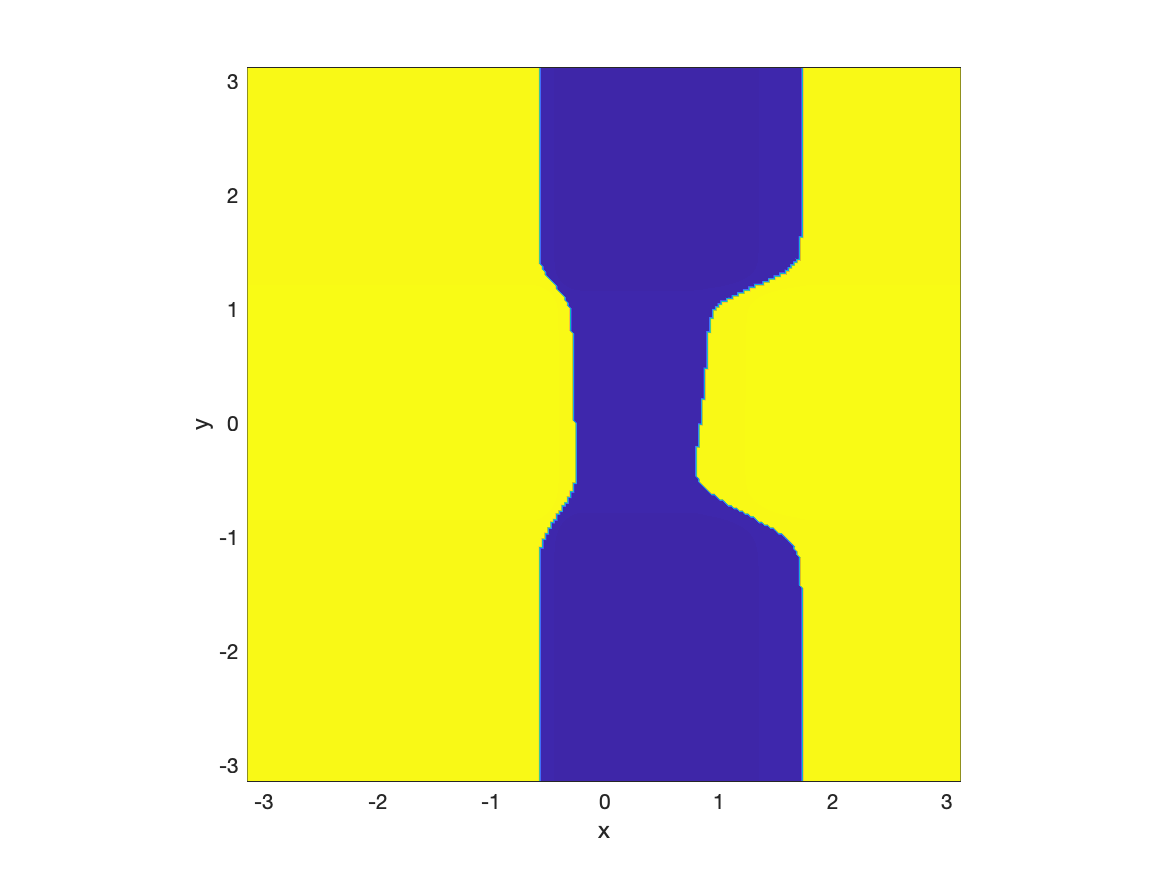} \\
\end{tabular}
\vskip-0.5cm
\caption{Multi-diffusive experiment. The underlying distribution $\rho(x,y)$ is composed of four equal Gaussians (see top left), distant from each other but two having similar $x$-coordinates for their centre and two having similar $y$-coordinates for their centre. The figures show $P(x,y,t)$ $t=5,10,40$ from top right to bottom right. In the colour scale, blue corresponds to $P=1$ and yellow to $P=0$ and therefore assignment to 2 clusters. Note that at the final time the clustering is according to proximity in one of the variables only. The background with constant low probability is assigned to minimize interfaces between $P=0$ and $P=1$.}

\label{fig:multidiff}
\end{center}
\end{figure}

\section{Summary}

This article develops a novel methodology for the unsupervised and semisupervised learning of categorical labels. In the limit of infinitely many sample points, the proposed algorithm converges to a set of reaction diffusion equations. Critical to the evolution of the system is an adaptive diffusivity that balances the reaction and diffusive terms through a parameter $\alpha$. The final state of the system experiences a fast transition in terms of this parameter, from largely unstructured rigid assignment to classes for $\alpha < 1$, to globally uniform assignments for $\alpha > 1$. For $\alpha \approx 1$, the evolution converges to rigid assignments to the various states, in regions separated by spatially smooth -- but vertically sharp-- interfaces. Unlike other reaction diffusion systems, where lacking non-uniform boundary conditions the interface between equilibria moves until one state dominates all others, the new system favors solutions converging to states of similar sized support.

The evolution ascends a Lyapunov function consisting of a regularized log-likelihood of the data. Notably, the data itself appears only in the regularizing, diffusive terms. Even though maximum likelihood motivates the development of the methodology and the construction of its Lyapunov function, its final, non-parametric form makes no assumption about the distributions underlying the data, bypassing density estimation altogether. The only external data required is one or more notions of distance between points. Thus the algorithm is indifferent to the dimensionality of the underlying space, scaling only with the number of sample points available. This scaling is made close to linear by keeping only near neighbors in the matrix that organizes the data into a weighted graph.

The possibility to aggregate various notions of distance makes the methodology very flexible. In particular, the article shows how it leads to a natural algorithm for the clustering of time series, and in a synthetic medical example how it addresses the curse of dimensionality in a natural way by dividing the variables among different diffusive networks. In the continuous limit, this results in a set of novel, non-local, diffusive operators, which relax the solution toward averages that are local in some dimensions but global in the others.


\bibliographystyle{plain}

\end{document}